\documentclass[11pt, a4]{article}

\usepackage{amsmath, color} \usepackage{amssymb, enumerate, amsthm}

\newcounter{kirshr}
\newcounter{kirsha}
\newcounter{kirshb}

\newenvironment{enumroman}{\setcounter{kirshr}{1}
\begin{list}{(\roman{kirshr})}{\usecounter{kirshr}} }{\end{list}}
\newenvironment{enumarab}{\setcounter{kirshb}{1}
\begin{list}{(\arabic{kirshb})}{\usecounter{kirshb}} }{\end{list}}

\newtheorem{theorem}{Theorem} \newtheorem{corollary}[theorem]{Corollary} \newtheorem{lemma}[theorem]{Lemma}
\newtheorem{conjecture}[theorem]{Conjecture} \newtheorem{proposition}[theorem]{Proposition}
 \newtheorem{definition}{Definition}
 \newtheorem{example}{Example}
\newtheorem{examples}[example]{Examples}  
\newcommand{\R}{\mathcal{R}}
\author{Tarek Sayed Ahmed}
\title{Hilbert's tenth problem, G\"odel's incompleteness, Halting problem, a unifying perspective}
 \renewcommand{\P}{\mathcal{P}} \newcommand{\nat}{\mathbb{N}}
\newcommand{\proj}{\pi^n_S}
 
\begin{document}
\maketitle
\begin{abstract}
We formulate a property $P$ on a class of relations on the natural
numbers, and formulate a general theorem on $P$, from which we get
as corollaries the insolvability of Hilbert's tenth problem,
G\"odel's incompleteness theorem, and Turing's halting problem. By slightly 
strengthening the property $P$, we get Tarski's definability
theorem, namely that truth is not first order definable. The
property $P$ together with a ``Cantor's diagonalization" process
emphasizes that all the above theorems are a variation on a theme,
that of self reference and diagonalization combined. We relate our results to self referential paradoxes, including a formalisation of the Liar paradox, 
and fixed point theorems.
We also discuss
the property $P$ for arbitrary rings. We give a survey on Hilbert's tenth problem for quadratic rings and for the rationals pointing the
way to ongoing research in main stream mathematics involving recursion theory, definability in model theory, algebraic geometry 
and number theory.
\end{abstract}
\section{Introduction}

In this paper we give a new unified proof to four important
theorems, solved in the period between 1930 and  1970. These problems
in historical order are G\"odel's incompleteness theorem,  Tarski's
definability theorem, Turing's halting problem \cite{3} and Hilbert's tenth
problem \cite{7}, \cite{Martin}. The latter was solved 70 years after it was posed by
Hilbert in 1900. Martin Davis, under the supervision of Emil Post,
was the first to consider the possibility of a negative solution to
this problem which ``begs for insolvability" as stated by Post. Later,
Julia Robinson, Martin Davis, Hilary Putnam worked on the problem for
almost twenty years, and the problem was reduced to showing that the
exponential function is Diophantine. This latter result was proved
by Matiyasevich in 1970.

The technique of proof of G\"odels incompletenes theorem, the Halting problem and Hilbert's tenth are quite similar.
In all cases we use G\"odel numbering to code meta statements as statements. So statements have a double role.
They are either theorems of the system, but subtly they can be viewed as meta statements, that is theorems about theorems of the system.
This line of thinking can be extended to obtain a theorem that refers to itself. 
We can code programs (for Halting problem), formulas (for incompleteness) and polynomials (for Hilbert's tenth problem).
In the Halting problem we use the notion of univeral Turing machine,
and for completeness we use the notion
of universal recursively enumerable set, and for Hilbert's tenth we use a universal Diaphantine set.
In all cases a universal set is defined from  which (by fixing one of the parameters) we obtain all the required sets
(for example we obtain all Diophantine functions or all listable\footnote{We may use the word listable instead of recursively enumerable}  subsets).
Then  we diagonalize out, so to speak. The ideas involve also strings that refers to strings,
formulas that refer to formulus possibly itself,
or programs that acts on numbers of other programs possibly the number of itself.
The idea of a string applied to its own G\"odel number, comes very much across in G\"odel's incompletenes theorem and the
Halting problem. The two ideas involved are diagonalization and self-reference, which can be reflected metamathematically by fixed point theorems. 
The diagonalization approach is best explained and exemplefied in \cite{y} and the self reference one in \cite{s}. 
The context of \cite{y} is category theory in disguise, while \cite{s} deals with abstract formal systems that are amenable to the Liar paradox. 
This paper attempts, on the one hand, to unify the two approaches,
and on the other, it stresses the fact that they are in a sense distinct.  In the famous Pulitzer prize winner book 
``G\"odel, Escher Bach" \cite{Ho}, a formal system
$TNT$ is used for arithmetic.  We quote:

``This chapter's title is an adaptation of the title of G\"odel's famous 1931 paper $=``TNT"$ being substituted for "Principia Mathematica".
G\"odel's paper was a technical one, concentarting on making his proof watertight and rigorous; this chapter will be more intuitive
and in it I will streess the two ideas that are in the heart of the proof.
The first key idea is the deep discovery that there are strings of $TNT$ which can be interpreted as speaking about other strings of $TNT$: in short: 
that $TNT$ as a language is capable of introspection or self scrutiny.
This is what comes from Godel numbering. The second idea is that the property of self scrutiny can be entirely concentrarted in a single 
string thus the string's sole  focus of attention is itself . The focusing trick is traceable to Cantor's diagonal method.

In my opinion if one is interested in understanding G\"odel's proof in a deep way then one must recognize that the proof, in essence, 
consists of a fusion of these two ideas.
Each one of them is a master stroke; to put them together is an act of genuis. If I were to choose, 
however, which of the two ideas is deeper I would unhesitatingly choose the first one, the idea of G\"odel numbering, 
for that notion is related to the whole notion of what meaning and reference are, in symbol manipulating systems. 
This is an idea that goes far beyond the confines of mathematical logic, wheras the Cantor 
trick, rich enough in its mathematical consequences, has little if any relation to issues in real life".

In what follows we give a unifying perspective of this phenomenon.
We formulate a property $P$ on a class of relations on the natural
numbers. We formulate a general theorem on $P$,
from which we get as corollaries the insolvability of the four
famous problems stated above. The property $P$ together with a
``Cantor's diagonalization" process emphasizes that all the above
theorems are a variation on a theme, that of self reference and
diagonalization combined. However, we argue that those two key concepts in G\"odels argument are in a sense independent from each other. 
We also discuss the property $P$ for
arbitrary rings; and an overview of Hilbert's tenth problem for other rings (like quadratic rings and the rationals), 
with emphasis on recent research, is presented.

The reader is assumed to be familiar with the known notions of recursive, recursively enumerable, Turing machines, and universal Turing machines.
A good reference is \cite{3}. Diophantine sets are defined in \cite{Martin}, and it is proved therein that recursively enumerable relations coincide
with Diaphantine sets. An excellent exposition of the proof of Hilbert's tenth problem is \cite{7}. 

No doubt that there is something about self referential paradoxes 
that appeals to all and sundry. 
And, there is also a certain air of mystery associated with them, but when people talk about such paradoxes in a non-technical 
fashion indiscriminately, especially when dealing with G\"odel's incompleteness theorem, then quite often it gets annoying!
Lawvere in {\it Diagonal Arguments and Cartesian Closed Categories} \cite{l} 
sought, among several things, to demystify the incompleteness theorem. 
In a self-commentary on the above paper, he actually has quite a few harsh words, in a manner of speaking.
``The original aim of this article was to demystify the incompleteness theorem of G\"odel and the truth-definition theory of 
Tarski by showing that both are consequences of some very simple algebra in the cartesian-closed setting. 
It was always hard for many to comprehend how Cantor's mathematical theorem could be re-christened as a paradox by Russell 
and how G\"odel's theorem could be so often declared to be the most significant result of the 20th century. 
There was always the suspicion among scientists that such extra-mathematical publicity movements 
concealed an agenda for re-establishing belief as a substitute for science."
In the aforesaid paper, Lawvere uses the language of category theory,  the setting is that of cartesian closed categories 
and therefore the technical presentation can easily get out of reach of most people's heads.

Thankfully, Noson S. Yanofsky has written a nice simple paper 
{\it A Universal Approach to Self-Referential Paradoxes, Incompleteness and Fixed Points} 
\cite{y} that is a lot more accessible and enjoyable. The idea is to use a single formalism, that uses no explicit category theory, to 
describe all these diverse phenomena of self referential paradoxes.
In this paper, we do a similar task, 
but in an even simpler approach, and we add Hilbert's tenth and the Halting problem to the list.
We show that the real reason for certain paradoxes, is that a certain system $g$ has limitations in a certain context. Transcending the boundaries
and allowing limitations, results in inconsistencies, that can be expressed
by fixed points. 
Hilbert's tenth is also surveyed for several rings reaching the boundaries of ongoing interdisciplinary 
research between recursion theory, model theory, algebraic geometry and 
number theory.
The paper is divided into two parts. The first part presents some new connections between self referential paradoxes, diagonalization,
and fixed point theorems, comparing them to previously known ones. The second part is more of an expositive character.
\section{Properties of relations on $\nat$}
\begin{definition}
A {\bf property} $P$ is defined to be a subset of
\;$\bigcup^{\infty}_{r=1}\P(\nat^r).$ We say that a set $X$
satisfies $P$ if $X \in P$ and a function $f:\nat^r\to \nat^s$
satisfies $P$ if as a relation it satisfies $P$.
\end{definition}
\begin{definition}
Let $n \in \nat, S \subseteq \{1,\dots,n\}$ and let $l$ be the
cardinality of $S$. We define the function $\pi^n_S:\nat^n\to\nat^l$
as follows $\proj(x_1,\dots,x_n)=(x_{s_1},\dots,x_{s_l})$ where
$S=\{s_1,\dots,s_l\},$ $1 \leq s_1<s_2<\dots<s_l\leq n.$
\end{definition}
\begin{definition}
Let $P$ be a property, then we say that $P$ is {\bf broad} if it
satisfies the following conditions.
\begin{enumerate}[(i)]
\item The sets $\nat^r, \{x\}$ satisfy $P$ for every $r\in \nat,x\in
\nat^r$.
\item The function of addition and multiplicationm $+, * :\nat^2\to\nat$, and the identity $\text{identity}:\nat\to\nat$ satisfy
$P$.
\item The property $P$ is closed under cross product and
intersection.
\item The function $\proj$ satisfies $P$ for every $n\in \nat,\ S\subseteq
\{1,\dots,n\}$.
\item For every $n\in \nat, \ S\subseteq \{1,\dots,n\}, V\subseteq
\nat^n$ satisfies P, $\proj(V)$ satisfies $P$.
\end{enumerate}
\end{definition}
\begin{examples}
(1)\quad Take $P=\bigcup^{\infty}_{r=1}\P(\nat^r).$\\
(2)\quad Take $P=\{D : D \text{ is a Diophantine set}\}$
\end{examples}
\begin{lemma} \label{inverse}
If $P$ is a broad property and $f:\nat^r\to\nat^s$ satisfies
$P$, $V$ satisfies $P$, then $f^{-1}(V)$ satisfies $P$.
\end{lemma}
\begin{proof}
If $P$ is a strong property and $f:\nat^r\to\nat^s$, $V$
satisfies $P$, then we claim that
$$f^{-1}(V)=\pi^{r+s}_{\{1,\dots,r\}}(f\cap(\nat^r\times V)).$$ To
prove the claim, we have $\overline{x} \in f^{-1}(V)$ iff
$f(\overline{x})\in V$ iff $\exists \overline{y}$ such that
$\overline{y}\in V$ ,$\overline{y}=f(\overline{x})$ iff $\exists
\;\overline{y}$ such that $\overline{y}\in V
, (\overline{x},\overline{y})\in f.$ Here we treat $f$ as a relation,
so $\overline{x}\in f^{-1}(V)$ iff $\exists \;\overline{y}$ such
that $(\overline{x},\overline{y})\in
f,(\overline{x},\overline{y})\in \nat^r\times V$ and this is
equivalent to $(\overline{x},\overline{y})\in f\cap(\nat^r\times V)$
which, in turn,  is equivalent to $\overline{x}\in
\pi^{r+s}_{\{1,\dots,r\}}(f\cap(\nat^r\times V))$ and hence the
claim is proved. Now by $(iii),(iv),(v)$ it follows that $f^{-1}(V)$
satisfies $P$.
\end{proof}
\begin{theorem}({\bf composition})\label{composition}
Let $P$ be a broad property, and the functions\newline
$h_1,\dots,h_n:\nat^k\to\nat$ satisfy $P$ and the function
$g:\nat^n\to\nat^s$ satisfies $P$, then the function
$f=g(h_1,\dots,h_n):\nat^k\to \nat^s$ satisfies $P$.
\end{theorem}
\begin{proof}
For simplification take $n=2$. By definition; to prove that $f$
satisfies $P$ we treat $f$ as a relation and prove that it satisfies
$P$, as was done in the previous lemma. We have
$(\overline{x},\overline{y})\in f$ iff $\exists \; t_1,t_2$ such
that $(\overline{x},t_1)\in h_1,(\overline{x},t_2)\in
h_2,(t_1,t_2,\overline{y})\in g$  is equivalent to
$(\overline{x},t_1,t_2,\overline{y})\in ((h_1\times \nat^{s+1})\cap
((\pi^{k+s+2}_{\{1,\dots,k,k+2\}})^{-1}(h_2))\cap(\nat^k\times g))$,
so we have\newline
$$f=\pi^{k+s+2}_{\{1,\dots,k,k+2,k+3,\dots,k+s+2\}}((h_1\times
\nat^{s+1})\cap
((\pi^{k+s+2}_{\{1,\dots,k,k+2\}})^{-1}(h_2))\cap(\nat^k\times
g)).$$ Since $h_1,h_2,g$ satisfy $P$, we have by lemma
\ref{inverse} and by $(iii),(iv),(v)$, that $f$ satisfies $P$
\end{proof}
\begin{corollary}\label{poly}
If $P$ is a broad property, then every polynomial in any
number of variables satisfies $P.$
\end{corollary}
\begin{proof}
The proof is divided into 3 parts.
First, the constant function satisfies $P$. Let $r,s\in \nat,c\in \nat^r$, then the function $f:\nat^s\to\nat^r$ defined as $f(x)=c$
for all $x\in\nat^k$ satisfies $P$. Because it is as a relation  
equal $\nat^s\times\{c\}$ which satisfies $P$ by $(i),(iii)$ it
follows by definition that $f$ satisfies $P$.\newline Second,  let
$r\in \nat$ then the function $f:\nat^k\to\nat^k$ defined by $f(x)=x$
for all $x\in\nat^k$ satisfies $P$ by $(ii)$.\newline Third, by
induction we can use part 1, 2, theorem \ref{composition} to show that every polynomial in any number of variables satisfies
$P$.
\end{proof}
\begin{theorem}\label{dio}
Let $P$ be any broad property, then every Diophantine set satisfies $P$.
\end{theorem}
\begin{proof}
Let $r\in\nat,\ D\subseteq\nat^r$ be a Diophantine set, then by
definition of Diophantine sets there exist a positive integer $m$
and a polynomial $Q(x_1,\dots,x_r,y_1,\dots,y_m)$ such that
$$\overline{x}\in D \Leftrightarrow \exists\;y_1,\dots,y_m,\text{
such that }Q(\overline{x},y_1,\dots,y_m)=0.$$ So $\overline{x}\in D
\Leftrightarrow \exists\; \overline{y} \text{ such that }
(\overline{x},\overline{y},0)\in Q$ which is equivalent to
$\overline{x}\in
\pi^{r+m+1}_{\{1,\dots,r\}}(Q\cap(\nat^{r+m}\times\{0\})).$ So
$$D=\pi^{r+m+1}_{\{1,\dots,r\}}(Q\cap(\nat^{r+m}\times\{0\}))$$ and
hence by corollary \ref{poly},$(iii),(v)$ it follows that $D$
satisfies $P.$
\end{proof}
\begin{theorem}\label{smallest property}
The set of all Diophantine sets forms  the smallest strong property.
\end{theorem}
\begin{proof}
To prove that the Diophantine sets form a strong property, we just
need to check the conditions.
\begin{enumerate}[(i)]
\item The zero polynomial proves that the set $\nat^r$ is
diophantine and the polynomial
$p(x_1,\dots,x_n)=\prod^n_{k=1}(x_k-c_k)$ defines the singleton
$c=(c_1,\dots,c_n)\in\nat^n$.
\item The addition, multiplication and identity functions are defined
by the polynomials $+(x,y,z)=x+y-z,*(x,y,z)=x*y-z,i(x,y)=x-y$
receptively.
\item Assume $A\subseteq\nat^r,B\subseteq\nat^s$ are Diophantine, then
there exist two polynomials $P,Q$ such that $\overline{x}\in A
\Leftrightarrow \exists\;y_1,\dots,y_m,\text{ such that
}P(\overline{x},y_1,\dots,y_m)=0$, $\overline{x}\in B\Leftrightarrow
\exists\;y_1,\dots,y_s,\text{ such that
}Q(\overline{x},y_1,\dots,y_l)=0,$ then the polynomial $PQ$ has the
property that $(\overline{x},\overline{y})\in A\times B
\Leftrightarrow \exists k_1,\dots,k_m,z_1,z_l \text{ such that }
P(\overline{x},\overline{k})\times
Q(\overline{y},\overline{z})=0$, so $A\times B$ is Diophantine, if
$r=s$ then the polynomial $P(x_1,\dots,x_r,y_1,\dots,y_m)\times
Q(x_1,\dots,x_r,z_1,\dots,z_l)$ defines the intersection, so again $A\cap
B$ is Diophantine.

\item The polynomial that defines the graph of the function $\proj$ is
$\prod^{l}_{k=1}(x_{s_k}-x_{m+k})$ where $S=\{s_1,\dots,s_l\},1 \leq
s_1<s_2<\dots<s_l\leq n.$
\item Let $A$ be a Diophantine set so there exist \newline a polynomial $P$
such that $\overline{x} \in A \Leftrightarrow
\exists\;y_1,\dots,y_m,\text{ such that
}P(\overline{x},y_1,\dots,y_m)=0$ so the set $\proj(A)$ is
Diophantine and defined by the polynomial $Q=P$ but with a
permutation of arguments \end{enumerate}
\end{proof}
\section{A unifying theorem (Diagonalization)}
\begin{theorem} \label{mainresult}
Let $P$ be a broad property such that $P\cap \P (\nat)$ is
countable. Let $D_1,D_2,\dots$ be an enumeration of subsets of
$\nat$ that satisfy $P$, then the function $g:\nat^2\to\nat$ defined
as the follows
\begin{equation*}
g(n,x) =
\begin{cases}
0 & \text{if } x\in D_n,\\
1 & \text{if } x\notin D_n.
\end{cases}
\end{equation*}
is not recursive, and $g$ does not satisfy $P$.
\end{theorem}
\begin{proof}
Assume $g$ satisfies $P$, then by theorem \ref{composition} and by
$(ii)$ it follows that the function $g(n,n):\nat\to\nat$ satisfies
$P$ and by lemma \ref{inverse}, the set $V=\{n|n\notin
D_n\}=g^{-1}(\{1\})$ satisfies $P$. Since $V\subseteq\nat$ satisfies
$P$ and the sequence of the sets $(D_n)_{n\in\nat}$ form an
enumeration of the subsets of $\nat$ satisfying $P$. So $\exists
\;n$ such that $V=D_n$ but then $n\in D_n \Leftrightarrow n\in V$
but by definition of $V$,we have $n\in V \Leftrightarrow n\notin
D_n$ And this is a contradiction.\newline So $g$ does not satisfy
$P$. Now  assume that $g$ is recursive, since a function is
Diophantine if and only if it is recursive \cite{Martin}, hence $g$
is Diophantine. By theorem \ref{dio} it follows that $g$ satisfies $P$
which leads to a contradiction. Hence $g$ is not recursive.
\end{proof}
\begin{theorem} (Pairing Function Theorem 1). There are Diophantine functions
$P(x,y)$, $L(z)$, $R(z)$ such that

(1) for all $x,y,L(P(x,y))=x$, $R(P(x,y))=y$, and

(2) for all $z,P(L(z),R(z))=z$, $L(z)\leq z,R(z)\leq z$.
\end{theorem}
\begin{proof}
See \cite{Martin} p. 203 Theorem 1.1.
\end{proof}
\begin{corollary}
Hilbert's tenth problem $(H10)$ is unsolvable
\end{corollary}
\begin{proof}
Take the following enumeration of polynomial with positive
coefficients
\begin{eqnarray*}
P_{1} &=&1 \\
P_{3i-1} &=&x_{i-1} \\
P_{3i} &=&P_{L(i)}+P_{R(i)} \\
P_{3i+1} &=&P_{L(i)}\cdot P_{R(i)}.
\end{eqnarray*}
where $L,R$ are the left, right functions respectively. Those are defined in \cite{Martin} p. 202. Finally, let
\[
D_{n}=\{x_{0}|(\exists
x_{1},\ldots,x_{n})[P_{L(n)}(x_{0},x_{1},\ldots, x_{n})=P_{R(n)}(x_{0},x_{1},\ldots,x_{n})]\}.
\]
then $D_n$ is an enumeration of Diophantine sets \cite{Martin} p.206. Assume that $H10$ is
solvable then the function
\begin{equation*}
g(n,x) =
\begin{cases}
0 & \text{if } x\in D_n,\\
1 & \text{if } x\notin D_n.
\end{cases}
\end{equation*}
is recursive which contradicts theorem \ref{mainresult}.
\end{proof}
\begin{corollary}(Halting problem)
There is no algorithm such that, given a program and an input to that
program, determines if the program halts at the given input or not.
\end{corollary}
\begin{proof}
Let $D_{n}=\{x|\text{the program number } n \text{ halt  at input }
x\}$ Then it is easy to see that $D_n$ form an enumeration of
listable sets, but listable sets are just Diophantine sets so by
theorem \ref{dio},\ref{mainresult} we get that the function
\begin{equation*} g(n,x) =
\begin{cases}
1 & \text{if } \text{the program number } n \text{ halt  at input }x\\
0 & \text{if } \text{otherwise}.
\end{cases}
\end{equation*}
is not recursive which is the required.
\end{proof}
\begin{corollary}\label{cor}
There exists a listable subset of $\nat$ which is not recursive.
\end{corollary}
\begin{proof}
In the previous proof and the unifying theorem proof we proved that
the function $g(n,n)$ is not recursive, so the set $g(n,n)^{-1}$ is
not recursive, but it is easy to see that it is listable.
\end{proof}
\begin{definition}
A relation $R\subseteq\nat^r$ on $\nat$ is said to be definable if
there exists a first order formula $\alpha(v_1,\dots,v_r)$ such that
$$\overline{x}\in R\Leftrightarrow\;\alpha(\overline{x}) \text{ is true in } \nat.$$
\end{definition}
\begin{examples}
Every Diophantine set is definable.
\end{examples}
\begin{corollary}(One version of G\"odel's Incompleteness Theorem)
Let $A$ be a recursive set of sentences true in $ \nat$, then there
exist a sentence $\sigma$ such that $\sigma$ is true in $\nat$ but
$A\nvdash\sigma.$
\end{corollary}
\begin{proof}
By theorem \ref{cor} there exists a listable set subset of $\nat$
which is not recursive. Let $K$ be that set. Since every listable
set is definable so the set $K$ is definable, that is there is a formula
$\chi(v)$ which defines $K$, then $\neg \chi(v)$ defines $K^c$.
Thus we have
$$a\in K^c\Longleftrightarrow \neg \chi(a) \text{ is true }$$
Let $A$ be a recursive set sentences true in $\nat$. Let
$J=\{a\in\nat | A\vdash\neg\chi(a)\}.$ Then we have
\begin{enumerate}

\item $J$ is recursively enumerable

\item $J\subseteq K^c$.

\end{enumerate}

We write $r.e$ instead of recursively enumerable. 
Now $J$ is a proper subset of $K^c$ because if $J=K^c$, then $K^c$ is
r.e and since $K$ is also $r.e$ so it follows that $K$ is recursive
which is a contradiction. Let $q\in K^c$ such that $q\notin J$. then
take $\sigma =\neg \chi(1)$ , then $q\in K^c$ says that $\sigma$ is
true in $\nat$ and $q\notin J$ says that $A\nvdash\sigma$ and the
theorem is proved.
\end{proof}
\section{A theorem of Tarski on definability}
\begin{definition}
A property $P$ is said to be {\bf very broad} if it is broad and
it is closed under complement.
\end{definition}
\begin{lemma}\label{use}
Let $P$ be a very broad property, $\alpha(v_1,\dots,v_r)$ be any
first order formula, then the set
$\{\overline{x}|\alpha(\overline{x})\text{ is true }\}$ satisfies
$P$
\end{lemma}
\begin{proof}
The proof is by straightforward induction on the complexity of formulas. Assume that $\alpha$ is an
atomic formula then the corresponding set satisfies $P$ by condition
$(ii)$. If $\alpha(v_{i\in I}),\beta(v_{k\in K})$ are two formulas
such that the sets $A=\{\overline{x}|\alpha(\overline{x})\text{ is
true }\},\ B=\{\overline{x}|\beta(\overline{x})\text{ is true }\}$
satisfy $P$ where $I,K$ are two finite subsets of $\nat$ then let
$l=|I|,s=|K|,j=|I\cup K|$. The two formulas
$\alpha\wedge\beta,\neg\alpha$ satisfy the lemma because the
corresponding sets are
$(\pi^{j}_{I'})^{-1}(A)\cap(\pi^{j}_{K'})^{-1}(B),A^c$
respectively, and those satisfy $P$ by lemma \ref{inverse} and $(iv)$,
where $I'=\{x-min(I)+1|x\in I\}$ the same for $K'$. The formula
$\exists v_k \alpha$ satisfies the lemma because its corresponding
set is $A$ if $k\notin I$ and it is $\pi^{l}_{m}(A)$ where
$I=\{x_1,\dots,x_m=k,\dots,x_l\}$ so it satisfies $P$ by lemma
\ref{inverse}. Hence all formulas satisfy the lemma.

\end{proof}
\begin{theorem} If $P$ is a very broad property, then every definable set satisfies $P$.
\end{theorem}
\begin{proof}
Let $V$ be any definable set, then there exists a formula $\alpha$
such that $$\overline{x}\in V \Leftrightarrow \alpha(\overline{x})
\text{ is true }.$$ So by lemma \ref{use} it follows that $V$
satisfies $P$
\end{proof}
\begin{theorem}\label{def}
The set of all definable sets form a very broad property.
\end{theorem}
\begin{proof}
To prove that the definable sets form a very broad property, we
just need to check that the conditions of very broad properties are
satisfied.
\begin{enumerate}[(i)]
\item The formula $v_1=
v_1\wedge\dots\wedge v_r=v_r$ defines the set $\nat^r$ so it is 
definable and the formula $v_1=S^{c_1}\wedge\dots\wedge v_n=S^{c_n}$ defines the singleton
$c=(c_1,\dots,c_n)\in\nat^n$ hence the first condition is satisfied.
\item The addition, multiplication and identity functions are defined
by the formulas $v_1+v_2=v_3,v_1*v_2=v_3,v_1=v_2$ receptively, hence
the second condition is also satisfied.
\item Assume $A\subseteq\nat^r,B\subseteq\nat^s$ are definable, then
there exist two formulas $\alpha,\beta$ such that $\overline{x}\in A
\Leftrightarrow \alpha(\overline{x})\text{ is true
},\;\overline{x}\in B\Leftrightarrow \beta(\overline{x})\text{ is
true }$, then the formula $\alpha(v_1,\dots,v_r)\wedge
\beta(v_{r+1},\dots,v_{r+s})$ has the property that
$(\overline{x},\overline{y})\in A\times B \Leftrightarrow
\alpha(\overline{x})\wedge \beta(\overline{y}) \text{ is true }$, so
$A\times B$ is definable, if $r=s$ then the formula
$\alpha(v_1,\dots,v_r)\wedge \beta(v_1,\dots,v_r)$ defines $A\cap B$
so the third condition is satisfied.
\item The formula that defines the graph of the function $\proj$ is
$\bigwedge^{l}_{k=1}(v_{s_k}=v_{m+k})$ where $S=\{s_1,\dots,s_l\},1
\leq s_1<s_2<\dots<s_l\leq n$, so the fourth condition is satisfied.
\item Let $V\subseteq\nat^n$ be a definable set so there exists a formula $\alpha$
such that $\overline{x} \in A \Leftrightarrow \alpha(\overline{x}).$
So the the set $\proj(A)$ is definable by the formula $\exists
v_{i\in I}\alpha(v_1,\dots,v_n)$ where $S=\{s_1,\dots,s_l\},1 \leq
s_1<s_2<\dots<s_l\leq n,I=\{1,\dots,n\}\backslash S.$
\item Let $V$ be a definable set so there exist a formula $\alpha$
such that $\overline{x} \in A \Leftrightarrow \alpha(\overline{x}),$
then the set $V^c$ is definable by the formula
$\neg\alpha(\overline{x}).$
\end{enumerate}
Hence the set of all definable sets forms a very broad property.
\end{proof}
\begin{theorem}\label{mainresult22}
Let $P$ be a very strong property and $P\cap\P(\nat)$ is countable,
So let $D_1,\dots$ be the enumeration of all subsets of $\nat$ and
satisfies $P$, then the set $V=\{(a,b)|a\in D_b\}$ is not definable.
\end{theorem}
\begin{proof}
Assume it is definable, since $P$ is closed under complement, then
$V^c$ is definable, hence the exist a formula $\alpha(v_1,v_2)$
defines $V^c$, so the formula $\alpha(v_1,v_1)$ defines the set
$Z=\{a|a\notin D_a\}$ but since $D_1,\dots$ form an enumeration of
all subsets of $\nat$ that satisfies $P$ but since $Z$ is a subset
of $\nat$ that satisfies $P$, so there exists $n$ such that $Z=D_n$
but by defintion of $Z$ $n\notin D_n\Leftrightarrow n\in Z$, so $Z$
cannot equal to $D_n$ which is a contradiction. Hence $V$ is not
definable.
\end{proof}
\begin{theorem}(Tarski definability theorem)

The set $\{(a,b)|\text{the statement } \alpha(v_1) \text{ with
G\"odel number $b$ is true at $a$ }\}$ is not definable.
\end{theorem}
\begin{proof}
Take $P$ be the class of all definable sets, then by theorem
\ref{def} this is a very broad property. Take
$D_b=\{a\in\nat:\text{the formula with G\"odel number $b$ is true at
a}\}$ then this is an enumeration of all subsets of $\nat$ that
satisfy $P$, so by theorem \ref{mainresult22}, the set $\{(a,b)|a\in
D_b\}$ is not definable, but this is precisely the conclusion of the theorem.
\end{proof}
In the paper 
Lawvere \cite{l} shows the G\"odel's incompleteness theorem  and the truth-definition theory of 
Tarski are consequences of some very simple algebra in the cartesian-closed setting. 
Here is his main theorem.

\begin{theorem}  In any cartesian closed category, if there exists an object $A$ and a weakly point surjective morphism
$$g:A\to Y^A$$
Then $Y$ has the fixed point property
\end{theorem}
In his nice simple paper  
``A Universal Approach to Self-Referential Paradoxes, Incompleteness and Fixed Points", 
Yanofsky presents the categorial result of Lawvere in \cite{l} employing  only the notions of sets and functions, 
thereby avoiding the language of category theory, to bring out and make accessible as much as possible the content of Lawvere's paper. 
(However categoriol concepts lurk behind the scene.) Cantor's theorem, Russell's Paradox, the non-definability of satisfiability, Tarski's non-definability 
of truth and G\"odel's (first) incompleteness theorem are all shown to be paradoxical phenomena 
that merely result from the existence of a cartesian closed category satisfying certain conditions. 
We now, in the same spirit of this paper on mixing levels, 
demystify Lawvere's theorem, that demystified G\"odel's incompleteness theorem,
by forgetting about the catogerical stuff and formulating his theorem using just maps and sets,
and relating it to our Diagonalizion lemma. The idea is to use a single formalism to 
describe many seemingly unrelated diagonalization arguments.
In this part, we closely follow \cite{y}. 
Let $g:T\times T\to Y$. $g$ is said to be universal if for all $f:T\to Y$ there exists $t\in T$ such that $g(s,t)=s$ for all $s\in T$.
In this case we say that $f$ is representable by $g$ and $t$ or simply by $g$. 
\begin{theorem}\label{d} Let $\alpha:2\to 2$ be the function defined by $\alpha(0)=1$ and $\alpha(1)=0$, 
then for all set $T$ and all functions $g:T\times T\to 2$, there exists a function $f:T\to Y$, such
that for all $t\in T$ 
$$f(-)\neq g(-, t).$$
\end{theorem}

\begin{proof} Let $\delta:T\to T\times T$ be the function
$$t\mapsto (t,t).$$
Then let $f=\alpha\circ g\circ \Delta$, that is 
$$g(t)=\alpha(f(t,t)).$$
Then clearly $g$ is as required.
\end{proof}

\begin{enumarab}

\item Our diagonalization lemma is an instance of the above theorem. Let $D_1,D_2,\dots$ be an enumeration of subsets of
$\nat$ that satisfy $P$, let $g:\nat\times \nat\to 2$ be defined by
\begin{equation*}
g(n,x) =
\begin{cases}
0 & \text{if } x\in D_n,\\
1 & \text{if } x\notin D_n.
\end{cases}
\end{equation*}
For each $n$, $g(-,n)$ is the characteristic function of $D_n.$
$$g(-,n)=\chi_{D_n}$$
$f$ as constructed of the theorem is the characteristic function of 
$$V=\{n\in \nat: n\notin D_n\}$$
and $f$ {\it cannot} be represented by $g$.

\item Cantor's diagonalization lemma. Let $D_1,D_2,\dots$ be an enumeration of subsets of
$\nat$ that satisfy $P$, let $g:\nat\times \nat\to 2$ be defined by
\begin{equation*}
g(n,x) =
\begin{cases}
0 & \text{if } x\in D_n,\\
1 & \text{if } x\notin D_n.
\end{cases}
\end{equation*}
For each $n$, $g(-,n)$ is the characteristic function of $D_n.$
$$g(-,n)=\chi_{D_n}$$
$f$ as constructed of the theorem is the characterstic function of 
$$V=\{n\in \nat: n\notin D_n\}$$
and $f$ cannot be represented by $g$.

\item Russell's paradox \cite{y} p.369. Let $g: Sets\times Sets\to 2$ be defined as follows
\begin{equation*}
g(s,t) =
\begin{cases}
1 & \text{if } s\in t\\
1 & \text{if } s\notin t.
\end{cases}
\end{equation*}
Then $g$ is the characteristic function of those sets that are not members of themselves.

\item Grelling paradox \cite{y} p.370. There are some adjectives in English that describe themselves 
and some that do not. For example ``English" is English and French is not a 
French word. Call words that do not apply to themeselves heterological. Now is ``heterological" heterological.
Let $g: Adj\times Adj\to 2$ be defned as follows
\begin{equation*}
g(s,t) =
\begin{cases}
 1 & \text{if } s\text { describes } t\\
1 & \text{if } s\text{ does not describe }  t.
\end{cases}
\end{equation*}
Then $f$ constructed as above is the charcteristic function of a property of adjectives that cannot be described by any adjectives.

\item Related to the above is the Liar paradox \cite{s}, and the strong Liar paradox and The Barber's paradox.
The solution of the last, is that there is a village where  everyone who does not shave themselves is shaved by the barber
simply does not exist.  A paper formulating the famous limitative theorems of G\"odel, Turing, Tarski and Church in a unified ``Liar formalism" is \cite{s},
the main ideas of which will be recalled in a while.
We should mention that the above paradoxes can be overcome as above. A related paradox is Quine's paradox
which is :

``yields falsehood when appended to its own quotation" yields 
falsehood when appended to its own quotation.

A common solution to the Liar paradox is saying that there are sentences 
that are neither true nor false but are meaningless ``I am lying" would be such a statement.
Consider now the sentence
[``yields falsehhod or meaninglessness when appended to it own quatotation" 
yields falsehhod or meaninglessness when appended to it own quatotation] 
This can be resolved by enlarging the set $2$ to $3$ and defining the function $g:Sent\times Sent\to 3$.
as follows:
\begin{equation*}
g(s,t) =
\begin{cases}
 T & \text{if } s\text { describes } t\\
M &\text{ if it is meaningless for $s$ to describe $t$}\\
F & \text{if } s\text{ does not describe }  t.
\end{cases}
\end{equation*}
Then $f$ is the characteristic function of sentences that are neither false nor meaningless 
when describing themselves.

\end{enumarab}
Extending further the previous definition, it is easy to see that theorem \ref{d} 
remains valid if $2$ is replaced by any set $Y$ and $\alpha:Y\to Y$ has no fixed points. Now taking the contrapositive of this, we 
get 
\begin{theorem}\label{d} Let $Y$ be  a set and there exists a set $T$ and a function $g:T\times T$ 
such that for all $f:T\to Y$ are representable by $f$, then every $\alpha:Y\to Y$ has a fixed point
\end{theorem}
\begin{proof} Like the above proof. See also \cite{y} Theorem 3
\end{proof}
Using this, we give a proof of the so called {\it Diagonalization} lemma in textbooks on G\"odel's incompletenes theorems. 
For a formula $\alpha(x), [\alpha(x)]$ denotes its G\"odel number.
We assume that we are working in a theory where there is a recursive function $D:\nat\to \nat$, such that 
$$D[\alpha(x)]=[\alpha([\alpha(x)])].$$ 
\begin{theorem} For any formula $\alpha(x)$ with $x$ as its only free variable, then there exists a sentence $\sigma$ such that
$$\vdash \sigma\leftrightarrow \alpha([\sigma]).$$
\end{theorem}
\begin{proof} \cite{y} Theorem 4. 
Let $Fm_n$ denote the set of formulas with free variables among the first $n$ and let $Sn$ denote the set of sentences.
Let $g:Fm_1\times Fm_1\to Sn$ be defined by 
$$(\mu(x), \beta(x))\mapsto \beta([\mu(x)]).$$
Let $\Phi_{\alpha}: Sn\to Sn$ be defned by 
$$\tau\mapsto \alpha([\tau])$$
Let $$f=\Phi\circ g\circ \Delta$$
Then $f$ is representable by $$g(x)=\alpha(D(x)).$$
So there is a fixed point of $\Phi_{\alpha}$ at $\sigma=g([g(x)]),$ which is the required.
\end{proof}
The best part of the schema worked out in theorem \ref{d} and exemplified above is that it shows 
that there are really no paradoxes. There are, rather, limitations. Paradoxes are ways of showing that violating a imitation leads to an inconsistency.
The Liar paradox shows that if you permit natural languages to talk about its own truthfulness 
then we will have inconsistencies. Russell's paradox shows that if we permit one to talk about any set without limitations, 
we will get an inconsistency. G\"odel's theorem shows that if we permit a system to talk about its truth then we get an inconsistency, however replacing 
truth by theoremhood
is a very rewarding thing to do, it gives the justly celebrated incompleteness theorems of G\"odel.
Though the provability relation is recursively enumerable, hence definable, Tarski showed that truth is not even definable.
G\"odel was smart enough to draw a line between theoremhood and truth, arriving at a positive limitative result, and not an inconsisteny. 
The above scheme in theorem \ref{d}
exhibits the inherent limitations of similar systems. The constructed or diagnolized out $f$ is a 
limitation that your system $g$ cannot deal with. If the system does deal with it, then there will be an inconsistency
(a fixed point $f$).

\section{Applications and Remarks}

\begin{enumarab}

\item Let $Prov(y,x)$ stand for $y$ is the G\"odel number of a proof of the statement whose G\"odel number is $x$.
Then let $$\alpha(x)=(\forall y)\neg Prov(y,x).$$
A fixed point for this $\alpha(x)$ in a consistent and $\omega$ consistent theory is a sentence that is equivalent to its own statement of unprovability.

\item Let $Neg:\nat\to \nat$ be defined for G\"odel numbers as follows
$$Neg[\alpha(x)]=[\neg \alpha(x)].$$
Let
$$\alpha(x)=(\forall y)(Prov(y,x)\to (\exists w)(w<y)\land Prov(w, Neg(x))).$$
A fixed point for this $\alpha(x)$ in a consistent theory is a sentence that is equivalent to its own statement of unprovability.
This s better known as Rosser's theorem.

\item Suppose that there exists a formula $T(x)$ that expresses the fact that $x$ is the G\"odel number
of a true theorem. Let $\alpha(x)=\neg T(x)$ , then a fixed point of $\alpha(x)$ does not do what is supposed to do. 
A theory in which diagonalization lemma holds cannot express its own theoremhood.

\item It is known that G\"odel's second incompleteness is obtained when the theory in question is strong enough to encode 
the proof of its first incompleteness theorem.
Such a theory cannot prove its consistency by methods formalizable in theory.
G\"odel's second incompleteness theorem says, officially, that given a set of axiom $A$  and rules by which you can deduce (prove) 
theorems from the axioms, if you can deduce all the laws of good old elementary arithmetic from the axioms $A$, then you can't prove that the axioms are 
consistent (i.e. that they aren't self-contradictory). 
George Boolos has  a great explanation of the second incompleteness theorem using only words of one syllable: 
First of all, when I say ``proved", what I will mean is ``proved with the aid of the whole of math. Now then: two plus two is four as you well know. 
And, of course, it can be proved  that two plus two is four (proved, that is, with the aid of the whole of math, 
as I said, though in the case of two plus two, of course we do not need the whole of math to prove that it is four). And, as may not be quite so clear, 
it can be proved that it can be proved that two plus two is four, as well. And it can be proved that it can be proved that it can be proved that two plus two is four. 
And so on . In fact, if a claim can be proved, then it can be proved that the claim can be proved. And {\it that } too can be proved. 
Now: two plus two is not five. And it can be proved that two plus two is not five. And it can be proved that it can 
be proved that two plus two is not five, and so on. 
Thus: it can be proved that two plus two is not five. Can it be proved as well that two plus two {\it is} five? 
It would be a real blow to math, to say the least, if it could. If it could be proved that two plus two is five, 
then it could be proved that five is not five, and then there would be no claim that could not be proved, and math would be a lot of bunk.
So, we now want to ask, can it be proved that it can't be proved that two plus two is five. 
Here's the shock: no, it can't. Or to hedge a bit: {\it if} it can be proved that it can't be proved that two plus two is five, {\it then}
it can be proved as well that two plus two is five, and math is a lot of bunk. In fact, if math is not a lot of bunk, then no claim of 
the form ``claim $X$ can't be proved" can be proved. 
So, if math is not a lot of bunk, then, though it can't be proved that two plus two is five, it can't be proved {\it that} it can't be proved that two plus two is five. 
By the way, in case you'd like to know: yes, it can be proved that if it can be proved that it can't be proved that two plus two is five, 
then it can be proved that two plus two is five. If this was a little too convoluted, check out 
G\"odel's Theorems, if you get it over there, I highly recommend coming back to read this--it's great fun. 
\item There is a difference between the fact that a sentence is provable and the fact that it is is provable that it is provable. 
This example can be found in Parikh's famous paper
{\it Existence and Feasability in Arithmetic}. We shall deal with following predicates:

$Prflength(m,x)$ $m$ is the length in symbols of a proof of a statement whose G\"odel number is $x$. This is clearly decidable.

$P(x)$ is $\exists yProv(x,y)$ there exists a proof of a statement whose G\"odel number is $x$.
and

$\alpha_n(x)$ is $\neg(\exists m<n Prflengh(m,x))$
Applying the diagonaliztion lemma to $\alpha_n(x)$ gives us a fixed point
$C_n$ that says ``I do not have a proof of myself shoretr than $n$"
There is a however a short proof of $P(C_n)$, as the reader is invited to show. Else he can consult \cite{y}.

\item We give a proof of Godel's second incompleteness theorem that fits into the diagonalization schema.
The proof is due to Jech.
That is we shall prove using a diagonalization argument that 
$ZF\nvdash Con(ZF).$ First note that $ZF$ proves the completeness theorem that is
$ZF\vdash  Con(ZF)\to ZF \text { has a model}$
Suppose, that $ZF$ proves that $ZF$ has a model.
Let $S$ be a finite set of axioms of $ZF$ enough to define models and satisfaction and 
contains one instance of the Comprehension axiom and proves that
$ZF$ has a model.
If $(M, E^M)$ $(N,E^N)$ are models of $ZF$ define $M<N$ iff there is some $(m, E^m)\in N$ such that $E^M=(E^m)^N=\{((x,y) N\models x E^{m} y\}$.
If $M<N$ then for sny sentence $\sigma$
$M\models \sigma$ iff $N\models m\models \sigma.$

The following can be easily proved.

(1) if $N\models S$ then there is $M<N$. 

(2) $<$ is a transitive relation.

Now if $\phi(x)$ is a formula, let $C_{\phi(x)}$ be the set of natural numbers defined by $\phi(x)$.
Let
$$D=\{\phi(x): \exists M(M\vdash S \& M\models \phi(x)\notin C_{\phi(x)}\}$$
Let $\theta(x)$ be the formulas 
$$\exists M(M\models S\& M\models x\notin C_x.$$
So $C_{\theta(x)}=D$.
Then $$S\vdash \theta(x)\in D \Leftrightarrow  \exists M(M\models S\& M\models \theta(x)\notin D.$$
Call $\theta(x)\in D$ $\theta$. Then $\theta$ plays the role of $\sigma$ in the Diagalinization  lemma.  
Then we can prove:

(3) if $N \models \theta$, then there is $M$ such that $M\models \neg \theta$.

(4) If $N\models \neg\theta$, then for every $M<N$ $M\models \theta$.

Now suppose $M_1\models S$. If $M_1\models \theta$, then by (3) there exists $M_2<M_1$ such that $M_2\models \neg\theta$.
Otherwise $M_2=M_1$. By (3) let $M_3<M_2$ then by (4) $M_3\models \theta$. By (3) let $M_4<M_3$ such that $M_4\models \neg\theta$. But by (2) 
$M_4< M_2.$
Contradiction.

\item 
There are other theorems in descriptive set theory which have the same flavour, that is diagonalizing out of a universal set.
In fact, the diagonalization trick also abounds in descriptive set theory and recursion theory. Here we give one example.
But first a few definitions \cite{j}. Let $X$ be a Polish space. A set $A\subseteq X$ is Borel if it belongs 
to the smallest $\sigma$ algebra of subsets of $X$ containing all closed sets.
For $\alpha<\omega_1$ we define the collections $\Sigma_{\alpha}^0$ and $\Pi_{\alpha}^0$ of Borel subsets of $X$:
$\Sigma_1^0=$ the collection of open sets.
$\Pi_1^0=$ the collection of all closed sets.
$\Sigma_{\alpha}^0=$ the collection of all sets
$A=\bigcup A_n$, where each $A_n$ belongs to $\Pi_{\beta}^0$ for some $\beta<\alpha$.
$\Pi_{\alpha}^0=$ the colletion of all complements of sets in $\Sigma_{\alpha}^0$.
 
\begin{theorem} For each $\alpha\geq 1$, 
there exists a set $U\subset N^2$ such that $U$ is $\Sigma_{\alpha}^0$ and for every $\Sigma_{\alpha}^0$ set $A$ in $N$, 
there exists $a\in N$ such that
$$A=\{x: (x,a\in U\}.$$
\end{theorem}
\begin{proof} \cite{j} Lemma 11.2
\end{proof}
\begin{theorem} For every $\alpha\geq 1$, there is a set $A\subseteq N$ that is $\Sigma_{\alpha}^0$ 
but not $\Pi_{\alpha}^0$.
\end{theorem}
\begin{proof} The proof is a typical diagonalization out of a universal set.
Let $U\subseteq N^2$ be a universal $\Sigma_{\alpha}^0$ set. Consider the set
$$A=\{x: (x,x)\in U\}$$
Clearly $A$ is a $\Sigma_{\alpha}^0$ set. If $A$ 
were also $\Pi_{\alpha}^0$, then its complement would be $\Sigma_{\alpha}^0$ and there would be an 
$a$ such that
$$A=\{x: (x,a)\notin U\}$$
Let $x=a$, then we obtain a contradiction.
\end{proof}
Similar results with the same techniques hold for the heirarchy of projective sets \cite{j} p.144. 
In an anlogous manner, diagonalization arguments are also used to distinguish between levels in the hierarachy of non computable sets.
A list of self referential paradoxes, including Rice's theorem, Richard's paradox, an oracle $B$ such that $P^B\neq NP^B$, time travel paradoxes,
Lob's paradox, the recrusion theorem, and Von Neuman's self reproducing machines are given in \cite{y}. 
It seems that the key idea of the diagonalization theorem \ref{d} 
is the existence of a recursive $D:\nat\to \nat$ which is the key to the fact that the system can talk about itself. 
A natural question is: can we find more abstract key properties in systems that make 
self reference possible in wider contexts. This is indeed done in \cite{s}. 
Indeed the approach in \cite{y} concentrates on diagonalistion, while that in \cite{s} concentrates
on self-reference. Both papers complement each other, each concentarting on one angle of the incompleteness
theorems. Now we review the main concepts in \cite{s}.
\begin{definition} $S=(E, S, F, N, g, s)$ is an abstract formal system if
\begin{enumroman}
\item $\emptyset\neq S\subseteq F\subseteq E$ and $N$ is an arbitrary set such that $F\cap N=\emptyset$.
$E$ is the set of expressions, while $F,S$ and $N$ are called the set of formulas, sentences and names respectively.
\item $g$ is a one to one function from $F$ to $N$. $g$ is the naming function. If $H\subseteq F$ then $\bold H$ will denote $g(H).$
\item $s$ is a mapping from $F\times N$ into $S$. $s$ is the substitution in $S$. We denote $s(\phi, n)$ by $\phi[n].$
\end{enumroman}
\end{definition}
Let $T\subseteq S$ be arbitrary. We consider $T$ to be the true sentences in $S$ , and sentences outside $T$
are the false ones. The result of substituting a linguistic phrase $q$ for the variable in $s$ by $s(q)$. 
The name of any expression will be denoted by
$[e]$ and we abbreviate the phrase
``The new sentence obtained by substituting the name of the sentence $x$ for the variable in it is false" by $p$.
Then we have the well known Findlay's paradox: 

``the new sentence obtained by substituting the name of the sentence 
`the new sentence obtained by substituting $x$ for the variabe in it is false' for the variable in it is false". 

This can be symbolically 
expresssed by  $f=p[p].$ In fact this is just another way of saying $``x$ is heterological" is heterological.
The formal name of the formula $x$ is $g(x)$ and the result of substituting this formal name in $x$ is $x[g(x)]$. Now $p$ corresponds to the sentence
$$x[g(x)]\notin T,$$ 
which we denote by $\bar{p}$. Then $\bar{p}$ is not an expression belonging to the formal system $S$, rather it is
{\it about} S. It is not a statement but indeed a metastatement. Though $\bar{p}$ does not belong to $S$, 
there might exist elements of $S$ that can in some way represent it, that is play the role of $\bar{p}$ within $S$.
So we are seeking a formula $\pi$, to the same effect as $\bar{p}$ which means that $\pi$ and $\bar{p}$ 
have to be true at the same time, they have to be true for exactly the
same formulas, that is by definition of $\bar{p}$  
for any $$\phi\in F,  \phi[g(\phi)]\in T\text { iff } \phi[g(\phi)]\notin T,$$
equivalently,
for any $$n\in N, \phi[n]\in T\text { iff } g^{-1}(n)[n]\notin T.$$
To formalize this, we need
\begin{definition} Let $S$ be a formal system and let $T\subseteq S$ , $X\subseteq N$ be arbitrary. 
We say that $X$ is $T$ representable in $S$, if there is a $\phi\in F$
such that for every $n\in N$,
$$\phi[n]\in T\leftrightarrow n\in X.$$
The formula $\phi$ is said to $T$ represent $X$ in $S$.
\end{definition}
Now if $\pi$ represents the set $\{n\in N: g^{-1}(n)[n]\notin T\}$, then we consider $\pi$ as a representative of $\bar{p}$ within $S$.
The ``Formalized liar" is the formal sentence $\pi[g(\pi)]$. This is true iff it is false implying that the formula $\pi$ does not exists. Formally
\begin{theorem}(Liar theorem) Let $S$ be an abstract formal system and $T\subseteq S$. 
Then the set $\{n\in N: g^{-1}(n)[n]\notin T\}$ is not $T$ representable. 
\end{theorem}
\begin{proof} Formalizing the above argument.
\end{proof}
Notice that $g^{-1}(n)[n]$ is just the formal version of the phrase ``the new sentence obtained by substituting $n$ 
into the sentence named by it", in other words $g^{-1}(n)[n]$ is exactly the formalization 
of {\it self reference}, it is roughly  the sentence whose G\"odel number is $n$ at the value $n$, clearly pointing to itself. 
Due to its central role in the paradox
it will appear in any formulation of the Liar paradox. We shall 
examine those formal systems that are strong enough, 
meaning that they contain enough formulas needed to express self reference within the system and are therefore amenable
to Cantor's diagonalization argument, reincarnated by G\"odel, proving undecidability.
\begin{definition} 
Let $S$ be an abstract formal system and let $A\subseteq S$, $X\subseteq N$. 
We say that $X$ is $A$ representable in $S$ if there a $\phi\in F$ such that for every $n\in N$
$$\phi[n]\in A\leftrightarrow n\in X.$$
$S$ is self referential with respect to $A$ if for any $X\subseteq N$,
the set
$$\{n\in N: g(g^{-1}(n)[n])\in X\}$$
is $A$ representable whenever $X$ is $A$ representable.
\end{definition}
These systems are {\it introspective}, indeed they can talk about themselves. 
Now we can formulate and indeed prove:
\begin{theorem}(generalized Liar theorem)
Let $S$ be an abstract formal system, $A\subseteq F$, $B\subseteq S.$ Assume that $S$ is self referential with respect to $B$
and suppose that $\sim A$ is $B$ representable. Then there is a $\lambda\in S$ such that
$$\lambda\in B \text { iff } \lambda\notin A$$
which implies
$$S\cap A\neq B.$$
\end{theorem}
\begin{proof} \cite{s} Theorem 1. Since $\sim \bold A$ is $B$ representable and $S$ is self-referential with respect to $B$, 
$\{n\in N: g(g^{-1}(n)[n])\notin A\}$
is also $B$ representable. There is a $\pi\in F$ such that for any $n\in N$, $\pi[n]\in B$ iff $g(g^{-1}(n)[n])\notin A$.
Let $m=g(\pi)$ and $\lambda=\pi[g(\pi)].$ Then, on the one hand, $\lambda\in S$. On the other hand, $\lambda=\pi[g(\pi)]\in B$ iff 
$\pi[m]\in B$ iff $g(g^{-1}(m)[m]\notin A$
iff $g^{-1}g(\pi)[g(\pi])\notin A$ iff $\lambda=\pi[g(\pi)]\notin A$. Then we have a $\lambda\in S$ such that
$\lambda\in B$ iff $\lambda\notin A$.
\end{proof}
Now to study formal logical systems we need further elaboration. For example we need to formulate negation.
That, following \cite{s}, is what we do in the next definition.
\begin{definition} Let $S$ be an abstract formal system.
\begin{enumarab}
\item Let $P,T\subseteq S$. We say that $L=(S,P,T,')$ is a logical system if $'$ is a mapping of $F$ into $F$
such that 
\begin{enumroman}
\item $\phi\in S$ iff $\phi'\in S$
\item $\phi\in T$ iff $\phi'\notin T$, whenever $\phi\in S$ and for every $n\in N$
\item $\phi[n]\in T$ iff $\phi'[n]\notin T.$
\end{enumroman}
$P$ and $T$ are called the set of provable and true sentences of $S$, and $\phi'$ is the negation of $\phi$.
\item Let $P,T\subseteq S$. 
\begin{enumroman}
\item $L$ is consistent if $P\cap P'=\emptyset$
\item  $L$ is complete if $P\cup P'=S$
\item $L$ is sound if $P\subseteq T$.
\end{enumroman}
\end{enumarab}
\end{definition}
Now the main theorems of G\"odel, Turing and Church on limitative results are obtained from the following 
\begin{theorem} Let $L$ be a logical system.
\begin{enumroman}
\item (a) Suppose that $L$ is self referential with respect to $P$ and $\bold P$ is $P$ representable. If $L$ is consistent, then $L$ is incomplete.

(b) Suppose that $L$ is self referential with respect to $T$ and $\bold P$ is $T$ representable. If $\L$ is sound, then $T\sim P\neq \emptyset$.

\item Suppose that $L$ is self referential with respect to $T$. Then
$\bold T$ is not $T$ representable.

\item Suppose that $L$ is self referential with respect to $P$.
Then
$\neg \bold P$ is not $P$ representable
\end{enumroman}
\end{theorem}
\begin{proof}  \cite{s}. We list the substitutions needed to obtain the different items from the Generalized Liar theorem:
\begin{enumroman}
\item (a) $A=\sim P, B=P'$

(b) $A=P, B=T$
\item $A=T, B=T$
\item $A=P, B=P$.
\end{enumroman}
\end{proof}
The statements of the theorem above can be considered to be the abstract versions of theorems of G\"odel, Tarski and Church 
concerning a single logical system. Indeed in its first main item consisting of the abstract syntactical versions 
of G\"odel's incompleteness theorem describing the relation between provability and refutability and that between provability
and truth. (a) says that strong enough systems that are self referential, if consistent, cannot be complete. They can construct sentences 
asserting their own unprovability.
(b) says that truth is strictly larger than provability, and this is due to the expressive strength of the proof system employed, not to its weakness,
there are always true statements that cannot be provable, no matter how we strenghten our system as long as we remain 
in the realm of the recursive.
This is expressed in our first version of G\"odel's theorem.
It contains, as its second and third main items, a generalization of Tarski undefinability of truth and 
a statement that can be seen as an abstract variant of Church's thesis 
on the undecidability of provability restricted to a single logical system. For more details the reader is referred to \cite{s}.

Other diagonalization arguments include Ascoli's theorem and the Baire category theorem in topology. 
The latter is strongly related to Martin's axiom and indeed
to forcing in set theory. 
One then could say that generic forcing extensions are formed by getting a ``diagonalized" out generic set, giving  a larger model.
This set is determined by the forcing conditions.

\end{enumarab}

However there are paradoxes which diagonalization cannot deal with.
Consider Richard's paradox. There are many sentences in English that describe real number between $0$ and $1$. 
Let us lexicographically order all English sentences. Using this
order, we can select all those English sentences that describe real numbers between $0$ and $1$.  
Call such a sentence a ``Ricahrd sentence". So we have the cncept of he $m$-th Richard sentence".
Consider the set $10=\{0,1,2,\ldots 9\}.$ 
and the function $\alpha:10\to 10$ defined as $\alpha(i)=9-i$. This Consider $f:\nat\times \nat\to 10$ defined as
$f(n,m)$ is the $n$th decimal number of the $m$th Richard sentence. Now consider $g:\nat\to 10$ constructed as $g=\alpha\circ f \circ \delta$.
Then $g$ describes a real number between $0$ and $1$ and yet for all $m\in \nat$ we have
$$g(-)\neq f(-,m).$$
This is because $g$ has no fixed point. However there is a Ricahrd sentence that describes this number:

`` $x$ is ther real number between $0$ and $1$ whose $n$th digit is nine minus the $n$ digit of the number described by the
$n$-Richard sentence"

Another example is the following: Let $F$ be the set of formulas and $g:F\to \nat$ be a Godel numbering. Let $[n]$ denote $g(n)$.
Note that $g^{-1}[n]$ is just $n$. Let $\sigma$ be any sentence. Consider the formula $\alpha(x)=g^{-1}x\to \sigma$.
A fixed point for $\alpha(x)$ is a $C$ such that
$\vdash C\leftrightarrow \alpha([C])\equiv g^{-1}[C]\to \sigma=C\to \sigma$.
So $C$ is equaivalent to $C\to \sigma$. By some reflection, we get that  $\sigma$ is always true. 

In \cite{l} an explanation is given.

\section{Generalization to other domains}

From now on we concentrate on Hilbert's tenth for various (recursive) rings. 
Let $\mathcal{R}$ be an integral domain. Then we can see that most of the
pervious proof dealng with Diaphantine definitions doesn not really  depend on the properties of $\nat$, so we
reformulate the theorems (concerning the property $P$) in any integral domain $\mathcal{R}$.\par
\begin{definition}
A {\bf property} $P$ is defined as a subset of
\;$\bigcup^{\infty}_{r=1}\P(\mathcal{R}^r).$ We say that a set $X$
satisfies $P$ if $X \in P$ and a function $f:\mathcal{R}^r\to
\mathcal{R}^s$ satisfies $P$ if it is as a relation satisfies $P$,
sometimes we say a set is $P$ to instead of a set  satisfies $P$.
\end{definition}
\begin{definition}
Let $n \in \nat,S \subseteq \{1,\dots,n\}$ and let $l$ be the
cardinality of $S$.We define the function
$\pi^n_S:\mathcal{R}^n\to\mathcal{R}^l$ as follows
$\proj(x_1,\dots,x_n)=(x_{s_1},\dots,x_{s_l})$ where
$S=\{s_1,\dots,s_l\},1 \leq s_1<s_2<\dots<s_l\leq n.$
\end{definition}
\begin{definition}
Let $P$ be a property. Then we say that $P$ is {\bf strong} if it
satisfies the following conditions.
\begin{enumerate}[(i)]
\item The sets $\mathcal{R}^r,\{x\}$ satisfy $P$ for every $r\in \mathcal{R},x\in
\mathcal{R}^r$.
\item The function $+,*,:\mathcal{R}^2\to\mathcal{R}$, $\text{identity}:\mathcal{R}\to\mathcal{R}$ satisfy
$P$.
\item The property $P$ is closed under cross product and
intersection.
\item The function $\proj$ satisfies $P$ for every $n\in \nat,S\subseteq
\{1,\dots,n\}$.
\item For every $n\in \nat,S\subseteq \{1,\dots,n\},V\subseteq
\mathcal{R}^n$ satisfies P , $\proj(V)$ satisfies $P$.
\end{enumerate}
\end{definition}
\begin{examples}
(1)\quad Take $P=\bigcup^{\infty}_{r=1}\P(\mathcal{R}^r).$\\
(2)\quad Take $P=\{D : D \text{is a Diophantine set in}
\mathcal{R}\}$
\end{examples}
\begin{lemma} \label{inverse2}
If $P$ is a strong property and $f:\mathcal{R}^r\to\mathcal{R}^s$
satisfies $P$, then the inverse image of every $P$ set is $P$.
\end{lemma}
\begin{proof}
If $P$ is a strong property and $f:\mathcal{R}^r\to\mathcal{R}^s,V$
is $P$,then we can prove as we did for the case $\R=\nat$ that
$$f^{-1}(V)=\pi^{r+s}_{\{1,\dots,r\}}(f\cap(\mathcal{R}^r\times V)).$$ It
follows that $f^{-1}(V)$ satisfies $P$.
\end{proof}
\begin{theorem}({\bf composition})\label{composition2}
Let $P$ be a strong property,and the functions\newline
$h_1,\dots,h_n:\R^k\to\R$ are $P$, the function $g:\R^n\to\R^s$ is
$P$, then the function $f=g(h_1,\dots,h_n):\R^k\to \R^s$ is $P$.
\end{theorem}
\begin{proof}
For simplification, take $n=2$. By definition; to prove that $f$
satisfies $P$ we treat $f$ as a relation and prove that it satisfies
$P$, as we have done with the case $\R=\nat$ we can prove that
$$f=\pi^{k+s+2}_{\{1,\dots,k,k+2,k+3,\dots,k+s+2\}}((h_1\times
\R^{s+1})\cap
((\pi^{k+s+2}_{\{1,\dots,k,k+2\}})^{-1}(h_2))\cap(\R^k\times g)).$$
Since $h_1,h_2,g$ satisfy $P$, we have by lemma \ref{inverse2} and
by $(iii),(iv),(v)$, that $f$ satisfies $P$
\end{proof}
\begin{corollary}\label{poly2}
If $P$ is a strong property, then every polynomial with coefficients
in $\R$ in any number of variables satisfies P.
\end{corollary}
\begin{proof}
The proof is the same as the proof in $\nat$. One proves that the
constant and the identity functions are $P$ then it follows by
induction that every polynomial with coefficients in $\R$ in any
number of variables satisfies $P$.
\end{proof}
\begin{theorem}\label{dio2}
Let $P$ be any strong property, then every Diophantine set satisfies $P$.
\end{theorem}
\begin{proof}
Let $r\in\R,D\subseteq\R^r$ be a Diophantine set, then by definition
of Diophantine sets there exist a positive integer $m$ and a
polynomial $Q(x_1,\dots,x_r,y_1,\dots,y_m)$ such that
$$\overline{x}\in D \Leftrightarrow \exists\;y_1,\dots,y_m,\text{
such that }Q(\overline{x},y_1,\dots,y_m)=0.$$ So $\overline{x}\in D
\Leftrightarrow \exists\; \overline{y} \text{ such that }
(\overline{x},\overline{y},0)\in Q$ which is equivalent to
$\overline{x}\in
\pi^{r+m+1}_{\{1,\dots,r\}}(Q\cap(\R^{r+m}\times\{0\})).$ So
$$D=\pi^{r+m+1}_{\{1,\dots,r\}}(Q\cap(\R^{r+m}\times\{0\}))$$ and
hence by corollary \ref{poly2},$(iii),(v)$ it follows that $D$
satisfies $P.$
\end{proof}
\begin{theorem}\label{smallest property2}
The set of all Diophantine sets form strong property.
\end{theorem}
\begin{proof}
To prove that the Diophantine sets in $\R$ form a strong property,
we just need to check the conditions.
\begin{enumerate}[(i)]
\item The zero polynomial proves that the set $\R^r$ is
Diophantine and the polynomial
$p(x_1,\dots,x_n)=\prod^n_{k=1}(x_k-c_k)$ defines the singleton
$c=(c_1,\dots,c_n)\in\R^n$.
\item the addition, multiplication and identity functions are defined
by the polynomials $+(x,y,z)=x+y-z,*(x,y,z)=x*y-z,i(x,y)=x-y$
receptively.
\item Assume $A\subseteq\R^r,B\subseteq\R^s$ are Diophantine, then
there exist two polynomials $P,Q$ such that $\overline{x}\in A
\Leftrightarrow \exists\;y_1,\dots,y_m$ such that
$P(\overline{x},y_1,\dots,y_m)=0$, $\overline{x}\in B\Leftrightarrow
\exists\;y_1,\dots,y_s$ such that
$Q(\overline{x},y_1,\dots,y_l)=0.$ Then since $\R$ is an integral
domain so the polynomial $PQ$ has the property that
$(\overline{x},\overline{y})\in A\times B \Leftrightarrow \exists
k_1,\dots,k_m,z_1,z_l \text{ such that }
P(\overline{x},\overline{k})\times
Q(\overline{y},\overline{z})=0$, so $A\times B$ is Diophantine, if
$r=s$ then since $\R$ is an integral domain the polynomial
$P(x_1,\dots,x_r,y_1,\dots,y_m)\times
Q(x_1,\dots,x_r,z_1,\dots,z_l)$ defines $A\cap B$.
\item The polynomial that defines the graph of the function $\proj$ is
$\prod^{l}_{k=1}(x_{s_k}-x_{m+k})$ where $S=\{s_1,\dots,s_l\},1 \leq
s_1<s_2<\dots<s_l\leq n.$ Here we used the fact that $\R$ is an
integral domain.
\item Let $V$ be a Diophantine set so there exist a polynomial \newline
$P$ such that $\overline{x} \in A \Leftrightarrow
\exists\;y_1,\dots,y_m,\text{ such that
}P(\overline{x},y_1,\dots,y_m)=0$ so the set $\proj(A)$ is
Diophantine and defined by the polynomial
$Q=P(x_{s_1},\dots,x_{s_l},y_1,\dots,y_{m+n-l})$ where
$S=\{s_1,\dots,s_l\},1 \leq s_1<s_2<\dots<s_l\leq n.$
\end{enumerate}
\end{proof}

For any integral domain, from the above it can be proved that any
Diophantine set is definable. However the converse seems to be
false, $Z$ is definable in $Q$, a result of Julia Robinson, but it
is conjectured that it is not Diophantine over $Q$ \cite{8}. In
quadratic rings a set is Diophantine if and only if it is listable
and self - conjugate \cite{2}. In particular Hilbert's Tenth
Problem is unsolvable for such rings. The analogous result for rationals
remains an open problem, and it seems a to be a hard one.

That Hilbert's Tenth Problem is 
unsolvable over $\mathbb{Z}$ was not the end; the solution of the original Hilbert's Tenth Problem opened up a new class of problems. 
The question posed by Hilbert could be asked for any other ring $R$: is there an algorithm that given any polynomial with coefficients in the ring 
$R,$ has solutions in $R$? This question has been answered for many rings, some examples follow.

\begin{itemize}
\item [$\mathbb{Z}$:] No algorithm exists (as proven in this paper)
\item [$\mathbb{C}$:] Yes, an algorithm exists (this was proved using elimination theory)
\item [$\mathbb{R}$:] Yes, an algorithm exists (this was proved by Tarski's elimination theory for semi-algebraic sets)
\item [$\mathbb{Q}$:] Still not solved!
\end{itemize}
Next we review the status of the problem on quadratic rings \cite{2}.

\section{Hilbert's Tenth Problem over $\mathbb{Z}[\sqrt{d}]$}

We shall prove 
that Hilbert's Tenth Problem is unsolvable for $\mathbb{Z}[\sqrt{d}]$, where $d$ is a square-free integer (we shall first 
confine ourselves to the case when d is positive). 
We will prove that there exists no algorithm that takes as input any 
Diophantine equation ``over $\mathbb{Z}[\sqrt{d}]$" and tells whether it has a solution in $\mathbb{Z}[\sqrt{d}]$ or not, 
by showing that assuming the solvability of the problem for $\mathbb{Z}[\sqrt{d}]$ implies its solvability for $\mathbb{Z}$ (which we know is not true).

\begin{definition}
A Diophantine equation over $\mathbb{Z}$ $(\mathbb{Z}[\sqrt{d}])$ is an equation of the form $P(\overline{x})=0$, 
where $P(\overline{x})$ is a polynomial with coefficients in $\mathbb{Z}$ $(\mathbb{Z}[\sqrt{d}])$, and we seek its solutions in $\mathbb{Z}$ $(\mathbb{Z}[\sqrt{d}])$.
\end{definition}

\begin{definition}
A subset $S$ of $\mathbb{Z}[\sqrt{d}]$ is said to be Diophantine over $\mathbb{Z}[\sqrt{d}]$ if there exists a polynomial $P(t,\overline{x})$ 
with coefficients in $\mathbb{Z}[\sqrt{d}]$ such that
$$t\in S \Leftrightarrow (\exists \overline{x})[P(t,\overline{x})]=0$$
\end{definition}

\begin{lemma}\label{lemma_1}
If $d^\prime$ is a square-free integer different from $d$, and $x,y\in \mathbb{Z}[\sqrt{d}]$, then $x=0$ and $y=0$ if and only if $x^2-d^\prime y^2=0$.
\end{lemma}

\begin{lemma}\label{lemma_2} (a theorem due to Lagrange)
Every natural number is the sum of four squares of natural numbers.
\end{lemma}

\begin{lemma}\label{lemma_3}
If $b\in \mathbb{N}$, and $b$ is not a square, then there exist natural numbers $x$ and $y$, with $y\neq0$ such that $x^2-by^2=1$.
\end{lemma}

\begin{definition}
For $a\in \mathbb{N}, a>1$ and $n\in \mathbb{N}$, we define $x_n(a),y_n(a)$ by setting:
$x_n(a)+(a^2-1)^{1/2}y_n(a)=(a+(a^2-1)^{1/2})^n$  $x_n(a),y_n(a)\in \mathbb{N}$.
\end{definition}
Where the context permits, the dependence on $a$ is not explicitly shown, writing $x_n,y_n$.
We note that $x_n$ and $y_n$ are studied in detail in \cite{Martin}, see the appendix therein where the following lemma is proved: 
\begin{lemma} \label{lemma_4}
If $b=a^2-1$, $a\in \mathbb{N}$, and $a>1$, 
then all the solutions in $\mathbb{N}$ of the Pell equation $x^2-by^2=1$ are given by $x=x_n(a), y=y_n(a)$.
\end{lemma}

\begin{lemma}\label{lemma_5}
If $a,n,k$ are in $\mathbb{N}$, and $a>1$, then $(y_{nk}(a))^2\equiv (y_n(a))^2k^2$ (mod $(y_n(a))^4$).
\end{lemma}
\begin{proof}
Consider the following,
\begin{align}
x_{nk}(a)+y_{nk}(a)(a^2-1)^{1/2} &=(a+(a^2-1)^{1/2})^{nk} \nonumber \\
&=(x_n(a)+y_n(a)(a^2-1)^{1/2})^k \nonumber \\
&=\sum_{j=0}^{k}\binom{k}{j}x_n^{k-j}y_n^j(a^2-1)^{j/2} \nonumber \\
y_{nk} &=\sum_{odd j;1\leq j\leq k} \binom{k}{j}(a^2-1)^{(j-1)/2}x_n^{k-j}y_n^{j} \nonumber \\
y_{nk}-kx_n^{k-1}y_n &= \sum_{j odd;3\leq j\leq k}\binom{k}{j}(a^2-1)^{(j-1)/2}x_n^{k-j}y_n^{j}.  \nonumber
\end{align}
Each term in this sum is divisible by $y_n^3$, hence $y_{nk}\equiv kx_n^{k-1}y_n$ (mod $y_n^3$).
and so $(y_{nk}/y_n)\equiv kx_n^{k-1}$ (mod $y_n^2$). Squaring yields $(y_{nk}/y_n)^2\equiv k^2(x_n^2)^{k-1}$ (mod $y_n^2$).

By Lemma \ref{lemma_4}  ,$x_n^2\equiv 1$ (mod $y_n^2$), thus $(y_{nk}/y_n)^2\equiv k^2$ (mod $y_n^2$), hence $y_{nk}^2\equiv y_n^2k^2$ 
(mod $y_n^4$).
\end{proof}

\begin{lemma}\label{lemma_6}
Let $a,b\in \mathbb{N}$ satisfy $a^2-db^2=1$, with $b\neq 0$ (such a $b$ exists because of lemma \ref{lemma_3}).
 Set $e=a^2-1$.If $x^2-ey^2=1$ and $x,y\in\mathbb{Z}[\sqrt{d}]$,then $y^2\in \mathbb{N}$.
\end{lemma}
\begin{proof}
We have $e=a^2-1=b^2d$, hence $x^2-b^2dy^2=1$. So $(x-b\sqrt{d}y)(x+b\sqrt{d}y)=1$. 
This means that $(x+b\sqrt{d}y)$ is a unit in $\mathbb{Z}[\sqrt{d}]$.
Put $u=(x+b\sqrt{d}y)$, then $u^{-1}=(x-b\sqrt{d}y)$, form which it follows that
$4b^2dy^2=(u-u^{-1})^2=u^2+(u^{-1})^2-2$, in other words, $u^2+(u^{-1})^2=4b^2dy^2+2$.

Since $u$ is a unit, $N(u)=\pm1$, that is, $u\overline{u}=1$, 
hence $\pm\overline{u}=u^{-1}$. And so $4b^2dy^2+2=u^2+(\overline{u}^2)$. Thus $4b^2dy^2+2\in \mathbb{Z}$. 
Hence $y^2\in \mathbb{Q}$. But $y^2\in \mathbb{Z}[\sqrt{d}]$, so $y^2\in \mathbb{Z}$; and since $d>1$, we have that $y^2\in \mathbb{N}$.
\end{proof}

\begin{lemma}\label{lemma_7}
Let $d>1$, and $x,y,z\in \mathbb{Z}[\sqrt{d}]$. If $x\equiv y$ (mod $z$) and $0\leq x<z$, $0\leq \overline{x}<\overline{z}$ ,
$0\leq y<z$, and $0\leq \overline{y}<\overline{z}$, then $x=y$.
\end{lemma}
\begin{proof}
Suppose that $x\neq y$, then $x-y=zw$, with $w\neq0$. So $|x-y||\overline{x}-\overline{y}|=|z\overline{z}||N(w)|.$
Since $w\neq0$, $|N(w)\geq1|$, hence $|x-y||\overline{x}-\overline{y}|\geq |z\overline{z}|$. But this is in contradiction with the hypothesis of the lemma.
\end{proof}

\begin{theorem}\label{theorem_1}
Let $\Sigma$ be the following system of Diophantine equations in the unknowns 
$t,x,y,u,v,z,w,h,q,r,s$ (where $e$ is as defined in the hypothesis of Lemma \ref{lemma_6}).
\begin{enumerate}
\item $x^2-ey^2=1$
\item $u^2-ev^2=1$
\item $v^2-y^2t=zy^4$
\item $t=w^2$
\item $y^2-t=1+h^2+q^2+r^2+s^2$
\end{enumerate}
Then
\begin{enumerate}
\item If $\Sigma$ has a solution in $\mathbb{Z}[\sqrt{d}]$, then $t\in \mathbb{N}$, and
\item If $k\in \mathbb{N}$ and $k\neq 0$, then $\sum$ has a solution in $\mathbb{Z}[\sqrt{d}]$ with $t=k^2$.
\end{enumerate}
\end{theorem}
\begin{proof}
\begin{enumerate}
\item Suppose that the system has a solution in $\mathbb{Z}[\sqrt{d}]$. 
From the first two equations and Lemma \ref{lemma_6}, we have $y^2,v^2\in \mathbb{N}$.
From 3 it follows that $(v^2/y^2)\equiv t$ (mod $y^2$).
$\overline{v^2-y^2t}=\overline{zy^4}$, hence $v^2-y^2\overline{t}=\overline{z}y^4$, and so $(v^2/y^2)\equiv \overline{t}$ (mod $y^2$). 
Thus $t\equiv \overline{t}$ (mod $y^2$).
By 4, we have $0\leq t$ and by 5 we have  $t<y^2$ and $y^2-\overline{t}=1+\overline{h}^2+\overline{q}^2+\overline{r}^2+\overline{s}^2,$ 
thus $0\leq t<y^2$ and $0\leq \overline{t}<y^2$.
Using Lemma \ref{lemma_7}, we obtain $t=\overline{t}$. So $t\in \mathbb{N}$.
\item Suppose $t=k^2,k\in \mathbb{N}$. Take $n$ such that $y_n(a)>k$, where $a$ is as in Lemma \ref{lemma_6}.
Set $x=x_n,y=y_n,u=x_{nk},v=y_{nk},w=k.$ By Lemma \ref{lemma_4}, 1 and 2 are satisfied. By Lemma \ref{lemma_5}, we can choose z 
such that 3 holds.
Obviously 4 is satisfied. Finally, $y^2-k^2>0$ and by Lemma \ref{lemma_2}, 5 can also be satisfied.
\end{enumerate}
\end{proof}

\begin{theorem}\label{theorem_2}
$\mathbb{N}$ is a Diophantine subset of $\mathbb{Z}[\sqrt{d}]$.
\end{theorem}
\begin{proof}
Put
$$P_1(x,y)=x^2-ey^2-1$$
$$P_2(u,v)=u^2-ev^2-1$$
$$P_3(v,y,z)=v^2-y^2t-zy^4$$
$$P_4(t,w)=t-w^2$$
$$P_5(y,t,h,q,r,s)=y^2-t-1-h^2-q^2-r^2-s^2$$
$$P(t,x,y,u,v,z,w,h,q,r,s)=(((P_1^2-d^\prime P_2^2)^2-d^\prime P_3^2)^2-d^\prime P_4^2)^2-d^\prime P_5^2$$
$Q_1=((P^2(k_1^2,x_1,y_1,u_1,v_1,z_1,w_1,h_1,q_1,r_1,s_1)-d^\prime P^2(k_2^2,x_2,y_2,u_2,v_2,z_2,w_2,h_2,q_2,r_2,s_2))^2-d^\prime$ 
$P^2(k_3^2,x_3,y_3,u_3,v_3,z_3,w_3,h_3,q_3,r_3,s_3))^2-d^\prime p^2(k_4^2,x_4,y_4,u_4,v_4,z_4,w_4,h_4,q_4,r_4,s_4)$
$$Q_2=t-k_1^2-k_2^2-k_3^2-k_4^2$$
$$Q=Q_1^2-d^\prime Q_2^2.$$
Where $d^\prime$ is a square-free integer different from $d$.

Suppose that there exists a solution for the equation $Q=0$, then by Lemma \ref{lemma_1} there exists a solution for each of
$P_1=0,P_2=0,P_3=0,P_4=0,P_5=0$. Then $t\in \mathbb{N}$ (by Theorem \ref{theorem_1}).

Conversely, if $t\in \mathbb{N}$, then there are $k_1,k_2,k_3,k_4$ such that the equation $Q_2=0$ is satisfied.
By Theorem \ref{theorem_1} and Lemma \ref{lemma_1}, there exists a solution for the equation $Q_1=0$.
Again by Lemma \ref{lemma_1} there exists a solution for $Q=0$.
\end{proof}

\begin{theorem}\label{theorem_3}
Hilbert's Tenth Problem is unsolvable for $\mathbb{Z}[\sqrt{d}]$.
\end{theorem}
\begin{proof}
Assume that there exists an algorithm that takes as input any 
Diophantine equation over $\mathbb{Z}[\sqrt{d}]$ and tells whether it has a solution or not. 
Let $P(a_1,\ldots,a_n)=0$ be a Diophantine equation over $\mathbb{Z}.$
There is a polynomial $Q$ such that
$$a_i\in \mathbb{N} \Leftrightarrow (\exists \overline{x_i}\in \mathbb{Z}[\sqrt{d}])[Q(a_i,\overline{x_i})=0].$$
Combine the following system of equations using a square-free integer $d^\prime$ 
different from $d$ (as was shown before) to obtain a polynomial $R$
$$P(\overline{a})=0$$
$$Q(a_1,\overline{x_1})=0$$
$$\vdots$$
$$Q(a_n,\overline{x_n})=0.$$
Apply the algorithm to $R$. If $R$ does not have a root in $\mathbb{Z}[\sqrt{d}]$, the algorithm tells us, 
but this means that for all possible choices of a tuple $(a_1,\ldots,a_n)$ in $\mathbb{N}$ there are $\overline{x_1},\ldots,\overline{x_n}$ 
in $\mathbb{Z}[\sqrt{d}]$ such that the equations
$$Q(a_1,\overline{x_1})=0$$
$$\vdots$$
$$Q(a_n,\overline{x_n})=0$$
together are satisfied, but the whole system
$$P(\overline{a})=0$$
$$Q(a_1,\overline{x_1})=0$$
$$\vdots$$
$$Q(a_n,\overline{x_n})=0$$
is not satisfied, which implies that all $(a_1,\ldots,a_n)$ in $\mathbb{N}$ does not satisfy $$P(\overline{a})=0.$$ 
Hence the algorithm tells us in this case that this equation does not have a solution.
If the algorithm tells us that $R$ has a root in $\mathbb{Z}[\sqrt{d}]$, then the whole system
$$P(\overline{a})=0$$
$$Q(a_1,\overline{x_1})=0$$
$$\vdots$$
$$Q(a_n,\overline{x_n})=0$$
has a solution(by Lemma \ref{lemma_1}). And the last $n$ equations tell us that $(a_1,\ldots,a_n)\in (\mathbb{Z}[\sqrt{d}])^n$.
Hence we have an algorithm that takes as input any Diophantine equation over $\mathbb{Z}$ 
and tells whether or not it has a solution in $\mathbb{N}$. 
But this implies (using the theorem due to Lagrange mentioned in Lemma \ref{lemma_2}) 
that Hilbert's Tenth Problem is solvable for $\mathbb{Z}$ (which is known to be false).
\end{proof}


\section{Hilbert's Tenth Problem over $\mathbb{Z}[i]$}

Here we are going to prove the unsolvability of Hilbert's Tenth Problem in the ring of Gaussian integers $\mathbb{Z}[i]$. 
Relying on its unsolvability in the ring of rational integers $\mathbb{Z}$, it is enough to prove that $\mathbb{Z}$ is
Diophantine in $\mathbb{Z}[i].$ For the proof we introduce some lemmas :
\begin{enumerate}[1.]
\item Consider the sequence $$\alpha(0), \alpha (1), \dots ,\alpha (n),\dots $$ 
defined by the following recurrence relation : $$\alpha(0)=0,\; \alpha(1)=1,\quad \alpha(n+1)=4\alpha(n)-\alpha(n-1).$$ 
This sequence is purely periodic modulo $4(3p^2+1)$. Since $\alpha(0)=0$ by definition, there are infinitely many multiples of 
$4(3p^2+1)$ in this sequence. From the recurrent relation, it is easy to show by induction that $$\alpha(n)\equiv n \pmod 2 .$$ 
Therefore, there is an odd value of $n$ such that $n>3$ and $$\alpha(n-1)\equiv 0 \pmod {4(3p^2+1)}.$$
\item All solutions in $\mathbb{Z}[i]$ of the equation $$x^2-4xy+y^2=0$$ 
are $$x=\alpha(n+1),\quad y=\alpha(n)$$ $$x=\alpha(n),\quad y=\alpha(n+1)$$ 
$$x=-\alpha(n+1),\quad y=-\alpha(n)$$ $$x=-\alpha(n),\quad y=-\alpha(n+1)$$where $n=0,1,2,\dots\;$.
\item $\alpha(pn)\equiv (-1)^pp\alpha(n-1)^{n-1}\pmod{\alpha(n)}$
\item For $a,b,c$ in $\mathbb{Z}[i]$ satisfying $a=b+cz$ and $|a|<|c|$, if $b$ and $c$ are rational integers, then $a$ is also a rational integer.
\end{enumerate}

\begin{theorem}
$\mathbb{Z}$ is Diophantine in $\mathbb{Z}[i] .$
\end{theorem}
\begin{proof}
We show that $a\in\mathbb{Z}$ if and only if a system of seven equations 
(in fact six are enough, the first equation is just for elegance) has a solution in $\mathbb{Z}[i].$\newline
The equations are :
\begin{enumerate}[(1)]
\item $2a+1=p$
\item$rx+t(8s+3)=1$
\item$x=4(3p^2+1)w$
\item$p=q+zy$
\item $v=qy$
\item$u^2-4uv+v^2$
\item$x^2-4xy+y^2$
\end{enumerate}
First let $a\in\mathbb{Z}.$ By 1 we can find an odd integer $n>3$ so that $\alpha(n-1)\equiv 0 \pmod{4(3p^2+1)}.$ Set $x=\alpha(n-1), y=\alpha(n).$
By 2 this is a solution of (7).\newline
Choosing $s$ such that $(x,8s+3)=1$ , there is $r$ and $t$ so that (2).\newline Set $u=\alpha(pn+1)$ and $v=\alpha(pn) .$ By 2 
this is a solution of (6).  We have that $y|v$, so there is $q$ so that (5). $x\equiv 0 \pmod{4(3p^2+1)}$ , so there is $w$ as in (3).\newline

Now form 3, $v\equiv p\alpha(n-1)^{n-1}\alpha(n)\pmod{\alpha(n)^2} $ which implies that $q\equiv px^{n-1} \pmod {\alpha(n)}$ 
from which $q\equiv p \pmod{\alpha(n)}$ since from (7) we have $x^2\equiv 1 \pmod y$ and so there is $z$ as in (4).
\newline \newline It is nice to notice that the obtained values are all integers .
\newline \newline
Conversely, if our system of equations has a solution in $\mathbb{Z}[i]$, then $a$ is an integer. Because any solution of (6),(7) in Gaussian integers is real 
(by 2) and so $q\in\mathbb{Z}$ . Since (2) has a solution, $x$ must be nonzero and so is $w$ (by (3)). \newline We have,
$$|x|=4|3p^2+1||w|>4|p|,$$ $$|x|=4|3p^2+1||w|\geq4,$$ which together with (7) imply that
$$|y|=|2\pm\sqrt{3+x^{-2}}||x|\geq(2-\sqrt{3+\frac{1}{16}})|x|=\frac{1}{4}|x|>|p|.$$ So lastly, by 4, 
$p$ must be an integer and so is $a$.

\end{proof}

\section{Hilbert's Tenth Problem over $\mathbb{Q}$}

Now the question is: does the negative answer for Hilbert's Tenth Problem over $\mathbb{Z}$ imply a 
negative answer for the same problem over $\mathbb{Q}$ (as was the case with quadratic rings)? 
If there is a Diophantine definition of $\mathbb{Z}$ in $\mathbb{Q}$, 
then the answer to the question is yes, we can transfer the results of Hilbert's Tenth Problem over $\mathbb{Z}$ to $\mathbb{Q}$. 
To prove this, we can use contradiction. Suppose Hilbert's algorithm exists for $\mathbb{Q}$. Then to solve an equation over $\mathbb{Z}$, consider the same equation over $\mathbb{Q}$ with auxiliary equations that ensure that the rational variables take integer values. Then Hilbert's algorithm will tell us whether this equation 
over $\mathbb{Z}$ has solutions in $\mathbb{Z}$. But this is a contradiction because Hilbert's Tenth Problem is unsolvable over the ring of integers.
Unfortunately, there is not yet a proof of whether a Diophantine definition of $\mathbb{Z}$ over $\mathbb{Q}$ exists. 
Barry Mazur has made a number of conjectures which, if true, 
imply that this definition does not exist. Mazur's two main conjectures are stated here \cite{5}.

\begin{conjecture}[Mazur's First Conjecture]
Let V be any variety over $\mathbb{Q}$. 
Then the topological closure of $V(\mathbb{Q})$ in $V(\mathbb{R})$ has at most finitely many connected components.
\end{conjecture}

\begin{conjecture}[Mazur's Second Conjecture]
There is no Diophantine definition of $\mathbb{Z}$ over $\mathbb{Q}$.
\end{conjecture}

It can be proved that the second conjecture is an implication of the first as shown below. Both conjectures are still unresolved.

\begin{proposition}
Mazur's First Conjecture implies Mazur's Second Conjecture.
\end{proposition}
\begin{proof}
We will prove this statement by contraposition. 
Suppose that $\mathbb{Z}$ has a Diophantine definition over $\mathbb{Q}$, 
this is the negation of the second conjecture. 
Then there exists a polynomial $p$ such that $$\mathbb{Z} = \{a \in \mathbb{Q}: (\exists x \in \mathbb{Q}^n) [p(a,x)=0]\}.$$ 
Let $X$ be the solution set of $p(t,x)=0$ for $t \in \mathbb{Z}$ in the affine space of dimension $1+n$. 
Then the topological closure of $X$ over $\mathbb{Q}$, $\overline{X(\mathbb{Q})}$, 
has at least one connected component above each $t \in \mathbb{Z}$. But this means that $\overline{X(\mathbb{Q})}$ 
has infinitely many connected components, which is the negation of the first conjecture. 
Thus, Mazur's first conjecture implies his second conjecture.
\end{proof}

\begin{definition}
Let $W_Q$ be a subset of the set of all primes $P$. Then the set $O_{Q,W_Q}:=\{x/y: x\in \mathbb{Z}, y\in <W_Q>\}.$
\end{definition}
We will now give a generalized form of Mazur's conjecture
\begin{conjecture}[Generalized Form of Mazur's Conjecture]
Let $W_Q \subset P$ where $P$ is the set of all primes. 
Let $V$ be a variety over $\mathbb{Q}$. Then $V(O_{Q,W_Q}) $ has finitely many connected components.
\end{conjecture}
However, the generalized form of Mazur's conjecture is not correct and we will give a counterexample. But first, we have to give a definition.
\begin{definition}
Let $W_Q$  be a subset of the set of all primes. Then we define Dirichlet density of $W_Q$ by \[\lim_{s \rightarrow 1^+}
\frac{\sum_{i\in \mathbb{N}}\frac{1}{p_i^s}}{|log(\frac{1}{1-s} )|}\] where $p_i\in A$.
\end{definition}
Now we will state the counterexample. However, we will not state the proof. The interested reader can consult \cite{5}.
\begin{proof}
For any $\epsilon >0$ there exist a set of rational primes $W_Q$ such that the Dirichlet density of $W_Q$ is greater than $1-\epsilon$ and there exist a variety V defined over $\mathbb{Q}$ such that the topological closure of $V_(O_{Q,W_Q})$ in $\mathbb{R}$ has infinitely many connected components.
\end{proof}
In his attempts to disprove Mazur's conjectures, 
Bjorn Poonen's results lead to a theorem, which carries his name. 
Poonen used the properties of elliptic curves to prove his theorem. Before stating this theorem, we will define two functions as:
\begin{align}
\pi(x) &= \text{The number of prime numbers } \leq x \nonumber \\
\pi_A(x) &= \text{The number of elements of A that are } \leq x. \nonumber
\end{align}
\begin{theorem}[Poonen's Theorem] \label{Poonen_th}
There is a computable set $A$ of prime numbers 
such that $$\lim_{x\rightarrow \infty} \frac{\pi_A(x)}{\pi(x)}=1$$ 
and there is a Diophantine definition of $\mathbb{Z}$ in $U$, where $U$ is the ring of all rational numbers whose denominators are products of primes in $A$. Thus Hilbert's Tenth Problem is unsolvable over the ring $U$.
\end{theorem}

So we have that $U \subseteq \mathbb{Q} $ and if $A=$ the set of all primes, then Hilbert's Tenth Problem is unsolvable over $\mathbb{Q}$. Poonen also succeeded in proving the following Lemma, which states that there is a map $f$ from $\mathbb{N}$ into $U$ such that the images of the graphs of the additive and multiplicative functions in $\mathbb{N}$ under $f$ are Diophantine over $U$.

\begin{lemma} \label{Poonen_l}
There is a computable function $f: n \mapsto y_n$ from $\mathbb{N}$ into $U$ (where $U$ is the same as in Theorem \ref{Poonen_th}) such that the sets of triples,
\begin{align}
\{(y_a,y_b,y_{a+b}) &|a,b \in \mathbb{N}\} \; \text{ and } \nonumber \\
\{(y_a,y_b,y_{ab}) &|a,b \in \mathbb{N}\} \nonumber
\end{align}
are both Diophantine over $U$.
\end{lemma}

Now let $S=\{a \in \mathbb{N} |(\exists x \in \mathbb{N}^k)[q(a,x)=0]\},$ where q is a polynomial with integer coefficients, then $S$ is a Diophantine definition of some recursively enumerable set. By Poonen's Lemma, since polynomials are built up by addition and multiplication of variables and constants, there must be some polynomial $p$ with coefficients in $U$ such that, $$S=\{a \in \mathbb{N} | (\exists x \in U^m)[p(y_a,x)=0]\}.$$
Let $S_p = \{a \in \mathbb{N} | (\exists x \in \mathbb{Q}^m)[p(y_a,x)=0]\}$. Since $U \subseteq \mathbb{Q}$, we have that $S \subseteq S_p \subseteq \mathbb{N}$.

\begin{definition}
A recursively enumerable (or listable) set $S \subseteq \mathbb{N}$ is \emph{simple} if $\mathbb{N}-S$ is infinite and contains no infinite recursively enumerable subset.
\end{definition}

If we apply the definition to a simple set $S$ then we have one of the following cases:
\begin{itemize}
\item $N-S_p$ is finite.
\item $N-S_p$ is infinite. If this is the case then $N-S_p$ is not recursively enumerable and does not contain any infinite recursively enumerable subset, because $(N-S_p)\subseteq(N-S)$ and $N-S$ is by definition infinite and contains no infinite recursively enumerable subsets. Thus, $S_p$ must be simple.
\end{itemize}

Recently, Martin Davis proposed the following conjecture, which if true will imply that Hilbert's Tenth Problem is unsolvable over $\mathbb{Q}$.
\begin{conjecture}
There is a Diophantine definition of a simple set $S$ for which $N-S_p$ is infinite.
\end{conjecture}

There are other approaches towards reaching a solution for Hilbert's Tenth Problem over $\mathbb{Q}$. 
Thanases Pheidas is attempting to construct a Diophantine model of $\mathbb{Z}$ with the aid of elliptic curves. 
While Cornelissen and Zahidi aim to update Julia Robinson's first order definability results.
Poonen has worked on many aspects of number theory, including the ``dark side" as he calls it. 
The dark side consists of results that show that something is undecidable. 
The problem is Hilbert's tenth for $\mathbb{Q}$ which is very simple to state, but seems to be very hard to prove. 
Hopefully, this problem can be solved
in the near future.
One must always recall the great Carl Gauss' quote: 
``I confess that Fermat's Theorem as an isolated proposition has very little interest for me, because I could easily lay down a multitude of such propositions, which one could neither prove nor dispose of. 
Number theory is filled with such innocent sounding problems that are impossible to resolve."
Perhaps Gauss was wrong about this particular example Fermat's Theorem turned out to be central to modern number theory,
but he is certainly right in general. Perhaps Fermat is ``the exception that proves the rule." 
An innocent open problem related to $H10$ for $\mathbb{Q}$ is: does there exist a polynomial map  $f(x,y)$ so that 
$f:\mathbb{Q}\times \mathbb{Q}\to \mathbb{Q}$
is a bijection? 
Harvey Friedman asked whether there exists a polynomial $f(x, y)$ in $Q[x, y]$ such that the induced map $\mathbb{Q} \times \mathbb{Q}\to \mathbb{Q}$ 
is injective. 
Heuristics suggest that most sufficiently complicated polynomials should do the trick. Don Zagier has speculated that a polynomial 
as simple as $x^7 + 3y^7$ might be an example. But it seems very difficult to prove that any polynomial works. 
Both Friedman's question and Zagier's speculation are at least a decade old, 
but surprisingly it seems that there has been essentially no progress on the question so far.
Poonen shows that a certain other, more widely-studied hypothesis implies that such a polynomial exists. 
Of course, such a polynomial does not exist if we replace $\mathbb{Q}$ by $\mathbb{R}$, the reals. In fact any injection 
$\mathbb{R}\times \mathbb{R}\to 
\mathbb{R}$ 
must be (very) discontinuous.
Suppose an injective polynomial $f$ could be identified, answering Friedman's question; 
it might then be interesting to look at `recovery procedures' to produce $x, y$ given $f(x, y)$, like done with the pairing function above. 
We can't possibly hope for $x, y$ to be determined as polynomials in $f(x, y)$, but maybe an explicit, 
fast-converging power series or some similar recipe could be found. 
Finally, all of this should be compared with the study of injective polynomials from the Euclidean plane to itself; 
this is the subject of the famous Jacobian Conjecture. See Dick Lipton's excellent post 
``http:rjlipton.wordpress.com/2010/07/17/an-amazing-paper" for more information.

We know that bijections exist between $\mathbb{Q}\times \mathbb{Q}$ and $\mathbb{Q}$ 
from George Cantor's famous work, since both  $\mathbb{Q}$ and  $\mathbb{Q}\times \mathbb{Q}$ are countable. 
The usual proof that  is countable arranges the rational numbers into a sequence of the form: 
$$0,1,-1,2,-2, 1/2,-1/2,3$$
The fractions $r/s$ are arranged based on the size of  $r$ and $s$. But, this sequence is quite messy, and is far from being definable by a polynomial. 
In a sense the shock of this question, let's call it the {\it Bijection Conjecture} $(BC)$, is that it is open at all. How can there be such a mapping? 
The method of Cantor clearly yields a very complicated function. How could a function from 
pairs of rationals map to single rationals and be a polynomial? Seems unlikely, but of course who knows. 
First raised by the famous logician Harvey Friedman, he also asked several related questions about other possible mappings. 
For instance, does there even exist a polynomial  that is one-to-one, let alone being bijective? 
This is believed for many polynomials such as $f(x,y)=x^7+ 3y^7$, but nobody has proved even one. 
These problems have enjoyed some recent interest, initiate by Poonen, \cite{rationals}. 
There are two simple consequences of the $BC$. The first shows that a bijection exists for any number of copies of $\mathbb{Q}$.
\begin{theorem}If the $BC$ is true, then for any $n\geq 2$ there is a bijection 
$f:\mathbb{Q}^n\to \mathbb{Q}$, where $f$ is a polynomial mapping
\end{theorem}
The second consequence concerns Hilbert's Tenth Problem and its potential generalization to rationals. 
One approach to proving that $H10$ is also undecidable for $\mathbb{Q}$ is to 
show that the integers are definable by a polynomial formula. 
This idea goes back to the early work of Julia Robinson, who proved in her Ph.D. thesis that the integers were 
definable in the first order theory of rationals. This proves that the first order theory of rationals is undecidable. 
The question is whether the restriction of that theory to formulas with existential quantifiers only, whose body is a single polynomial 
set equal to zero, is likewise undecidable. 
Robinson's undecidable formulas had several levels of alternating  $\forall$ and  $\exists$ quantifiers,
making them relatively complex. Over the years there have been improvements, and the best known is due to Poonen. 
Let's consider formulas of the form: a series of prenex quantifiers and then a polynomial equation. 
Let  $\Sigma$ denote those where all the prenex quantifiers are existential, and let  $\Pi$ where they all are universal quantifiers. 
We can mix these as usual; thus, $\Sigma\Pi\Sigma$
denotes first existential, then universal, and last existential. We will also allow, for example, 
$\Sigma\Pi^2\Sigma$
to denote formulas where the block of universal are just two quantifiers. 
In this notation, Poonen proves: 

\begin{theorem}There is a $\Pi^2\Sigma$  formula that defines the integers. 
\end{theorem}
Thus, formulas of the above are undecidable.
But these are still complicated.
We will just show what these  formulas look like.
Consider the following definitions: 
$$L=(a+x_1^2+x_2^2+x_3^2 +x_4^2)(b +x_1^2+ x_2^2+x_3^2+x_4^2)\Delta$$
where 
$$\Delta=(x_1^2-ax_2^2-bx_3^2 +abx_4^2-1) +\prod_{n=0}^{2309}\left((n-x-2x_1)^2-4ay_2-4b y_3^2 +4aby_4^2.4\right)^2$$
Define  $\psi(x)$ to be 
$$\forall a \forall b\exists x_1\ldots \exists x_4\exists y_2\ldots \exists y_4 (L=0)$$
Then Poonen's main theorem is: 
\begin{theorem} For all rational, the formula  $\psi(x)$ is true if and only if  $x$ is an integer.
\end{theorem} 
Poonen makes clever use of quaternions, and properties of various Diophantine equations. 
There should be other simple consequences of this strong assumption. 
Can we find other applications? 
The ability to map  tuples of rationals to one rational, 
by a polynomial map, should have other applications and repercussions in computational complexity.

The definition of Diaphantine sets naturally identifies these sets as number theoretic objects. 
Matiyasevisch's theorem tells us that these sets also belong to recursion theory. However, 
one can consider Diaphantine equations as sets definable in the language of rings 
by positive existential formulas and thus a subject of Model Theory. 
Finally Diaphantine sets are also projections of algebraic sets 
and consequently belong to algebraic geometry. In \cite{5} a pronounced bias towards 
number theoretic view of the problem is displayed, 
but some forays in geometry are made through the results of Mazur and Poonen.

Strikingly, there are many problems reductions of which to the $10$  problem for $\mathbb{Z}$ are not difficult but just much less evident. 
We start from the famous Fermat's Last Theorem . Hilbert did not include explicitly this 
problem in his List of Problems. Formally the problem is about unsolvabiltey of an infinite sequence of Diophantine equation
$$x^n+y^n=z^n$$
and thus it is not a case of the Hilbert's 10th problem in which Hilbert asks for solving only one single 
Diophantine equation rather than an infinitely many.
Fermat's equation is Diaphontine in $x,y,z$ for a fixed value of $n$ but is an exponential Diaphontine equation if viewed 
as an equation in four unknowns $n,x,y,z$. 
But, by the work of Matiysavech we know how to transfer an arbitary exponential Diophantine equation into genuine Diaphontine equation with 
extra varibales.
We can construct a particular polynomial $F$ with integer coefficients such that equation
$$F(n,x,t,z, u_1\ldots u_m)=0$$ has a solution in $u_1,\ldots u_m$ if and only if $n,x,y,z$ are solutions. 
So Fermat's last theorem is equivalent to the stateent that a particular Diaphontine equation has no solution in non negative integers. 
Thus a positive solution of the 10 poblem in its original formulation give us a tool to prove or disprove Fermat's Last Theorem. 
So while Fermat's last theorem is not explicitly among Hibert's problems it is presented there implicitly as a very particluar case of $H10$.
In spite of the fact that such such a reduction of Fermat's last theorem to a fixed Diaphantine equation was not known before 1970, is not too striking .
As a less evident example we can consider another famous problem Goldbach's Conjecture which was included by 
Hilbert into the 8th problem and still remains open.
Goldbach's Conjecture is very simple to formuate; it states that every integer greater than $4$ is the sum of two primes. For a particular number 
$a$ we can check whether it is a counterexample to Goldbach's conjecture or not. Thus the list $C$ of counterexamples
is listable and hence Diophantine . Respectively, we can find a particular Diaphantine equation
$$G(a,x_1,\ldots x_n)=0$$
which has a solution if and only if $a$ spoils the conjecture. In other words Goldbach's Conjecture is equivalent to the statement that the 
set $C$ is empty and hence to the Diophantine equation
$$G(x_0,\ldots x_m)$$
has no solution at all.
Thus we see again that positive solutions to the Tenth Problem in its original form allows us to know that Goldbach's Conjecture is true or not.
The reduction of Goldbach's conjecture to particular Diaphantine equations 
is less evident and requires more techniques than the reduction of Fermat's last theorem because now we have to deal with primality. 
Still it is believable because Golbach's conjecture is after all about integers.
Besides Goldbach's Conjecture, Hilbert included into the 8 problem another outstanding conjecture: the famous Reimann Hypothesis. 
In its original formulation it is a statement about
complex zeros of Reimann's zeta function which is analytic and has the form $$\xi(z)=\sum 1/n^s$$
which converger for $R(z)>1$. Nevertheless we can also construct a particular Diaophontine equation 
$$R(x_1,\ldots x_n)$$
which has no solution if and only if the Reimann Hypothesis is true. 
Such a reduction requires either the use of the theory of complex numbers or the use of the fact
that the Reimann hypothesis can be reformulated as a statement about the distribution of prime numbers.
Thus once again we see that an outstanding mathematical problem is a specific case of Hilbert's tenth.

The above three statements i.e Fermat's Last Theorem, Goldbach's conjecture, and the Reimann Hypothesis 
were about numbers and so their reduction to Diophantine equations is imaginable. 
However, in mathematical logic the powerful tool of arithmetizaton (used by G\"odel to code statements by numbers, making his system introspective)
allows one to reduce to numbers 
many problems which are not about numbers at all.
As an example consider yet another famous challenge to mathematics 
is the Four Color Conjecture which since 1976 is a theorem of K Apel and W, Haken. 
This is a problem about
colouring planar maps but again we can construct a Diaphantine equation 
$$C(x_1,\ldots x_m)$$
which has no solution  if and only if the Four Color problem is true. 
Again a problem which was not included by Hilbert into the Problem, 
appears in a masked form in the 10th probem.

Now two of the above problems are now solved, the two others remain open. The reductions of these problems may 
be considered as striking, stunning, amazing but could these reductions be useful at all ? Hilbert's tenth is undecidable
so we do not have a universal method to solve all these problems at once.
We cannot solve any of these problems by looking at particular Diophantine Equation  because they are even more complicated.

But quoting Matiyasevich : ``We can reverse the order of things. The 10th  problem is undecidable and we need to 
invent more and more adhoc methods to solve more and more Diophantine equations. 
Now we can view the proof of Fermat Last theorem and the four color conjecture as 
very deep tools for treating particular Diaphantine equations and we may extend such techniques to other Diophantine equations."

We should emphasize that Hilbert's tenth for $\mathbb{Q}$ is a central problem in main stream mathematics.  
It is related to sophisticated logic machinery through undecidability and recursion theory, to number theory via elliptic curves, 
to theory of definability in model theory 
which is taking a new sophisticated geometric twist
by the recent work of Hroshovski relating notions in stability theory (like Morely rank) 
to geometric notions in abelian varieties,
solving the famous Mordell Lang conjecture \cite{h}.
There is an interesting phenomenon arising here. The Mordell Lang conjecture is a conjecture that comes from algebraic geometry. 
It was solved by advanced techniques originally invented by Sharon Shelah in model theory. On the 
other hand Hilbert's tenth for $\mathbb{Q}$ is a problem that comes from mathematical logic, but it seems likely that it will be solved by advanced 
algebraic geometry machinery. One, then, naturally wonders whether Huroshovski's techniques are applicable in the ``Hilbert tenth context",
both basically dealing with Diaphantine geometry.

\end{document}